%% file: cle_geodesic_uniqueness_arxiv20251204.tex
\documentclass[11pt]{amsart}
\usepackage[utf8]{inputenc}
\usepackage[T1]{fontenc}
\usepackage{lmodern}
\usepackage{microtype}
\usepackage[leqno]{amsmath}
\usepackage{amssymb}
\usepackage{mathtools}
\usepackage{enumerate}
\usepackage{graphicx}
\usepackage[usenames,dvipsnames]{xcolor}
\usepackage[
colorlinks=true,linkcolor=NavyBlue,urlcolor=RoyalBlue,citecolor=PineGreen,bookmarks=true,bookmarksdepth=3,bookmarksopen=true,bookmarksopenlevel=2,unicode=true,linktocpage]{hyperref}

\numberwithin{equation}{section}
\numberwithin{figure}{section}

\newtheorem{theorem}{Theorem}[section]

\newtheorem{lemma}[theorem]{Lemma}
\newtheorem{proposition}[theorem]{Proposition}

\newtheorem{definition}[theorem]{Definition}

\newtheorem{assumption}[theorem]{Assumption}

\let\C\relax

\newcommand{\C}{\mathbf{C}}
\newcommand{\D}{\mathbf{D}}

\newcommand{\h}{\mathbf{H}}
\newcommand{\N}{\mathbf{N}}
\newcommand{\Z}{\mathbf{Z}}
\newcommand{\p}{\mathbf{P}}
\newcommand{\Q}{\mathbf{Q}}
\newcommand{\R}{\mathbf{R}}

\newcommand{\CC}{\mathcal {C}}

\newcommand{\CF}{\mathcal {F}}
\newcommand{\CI}{\mathcal {I}}
\newcommand{\CJ}{\mathcal {J}}
\newcommand{\CK}{\mathcal {K}}
\newcommand{\CL}{\mathcal {L}}

\newcommand{\CQ}{\mathcal {Q}}

\newcommand{\CS}{\mathcal {S}}

\newcommand{\CX}{\mathcal {X}}

\newcommand{\CG}{\mathcal {G}}

\newcommand{\SLE}{{\rm SLE}}
\newcommand{\CLE}{{\rm CLE}}

\newcommand{\dist}{\mathrm{dist}}

\newcommand{\diam}{\mathrm{diam}}

\newcommand{\re}{\mathrm{Re}}

\newcommand{\one}{{\bf 1}}

\newcommand{\wt}{\widetilde}
\newcommand{\wh}{\widehat}
\newcommand{\ol}{\overline}
\newcommand{\ul}{\underline}

\newcommand{\quant}[3][]{{{\mathfrak q}}_{#3}\if\relax\detokenize{#1}\relax\else^{#1}\fi(#2)}
\newcommand{\median}[2][]{{\mathfrak m}_{#2}\if\relax\detokenize{#1}\relax\else^{#1}\fi}
\newcommand{\mediant}[2][]{\wt{\mathfrak m}_{#2}\if\relax\detokenize{#1}\relax\else^{#1}\fi}

\newcommand{\Fd}{\mathfrak d}
\newcommand{\met}[3]{\Fd(#1,#2;#3)}
\newcommand{\mett}[3]{\wt{\Fd}(#1,#2;#3)}
\newcommand{\metres}[4]{\Fd^{#1}(#2,#3;#4)}
\newcommand{\mettres}[4]{\wt{\Fd}^{#1}(#2,#3;#4)}

\newcommand{\bIn}{\partial_{\mathrm{in}}}
\newcommand{\bOut}{\partial_{\mathrm{out}}}

\DeclarePairedDelimiter\abs{\lvert}{\rvert}

\newcommand*{\defeq}{\mathrel{\mathop:}=}

\newcommand*{\mmiddle}[1]{\mathrel{}\middle#1\mathrel{}}

\newcommand*{\Fill}{\operatorname{fill}}

\newcommand*{\sle}[1]{$\SLE_{#1}$}
\newcommand*{\slek}{\sle{\kappa}}
\newcommand*{\slekp}{\sle{\kappa'}}
\newcommand*{\slekr}[1]{$\SLE_{\kappa}(#1)$}
\newcommand*{\slekpr}[1]{$\SLE_{\kappa'}(#1)$}
\newcommand*{\cle}[1]{$\CLE_{#1}$}
\newcommand*{\clek}{\cle{\kappa}}
\newcommand*{\clekp}{\cle{\kappa'}}

\newcommand{\domainpair}[1]{{\mathfrak {P}}_{#1}}

\newcommand{\outside}{{\mathrm{out}}}
\newcommand{\inside}{{\mathrm{in}}}

\newcommand{\resampled}{{\mathrm{res}}}

\newcommand{\dGHf}{d_{\mathrm {GHf}}}

\newcommand*{\metregions}[1][]{\mathfrak{C}\if\relax\detokenize{#1}\relax\else_{#1}\fi}

\newcommand*{\distE}{\operatorname{dist_E}}
\newcommand*{\diamE}{\operatorname{diam_E}}
\newcommand*{\dpath}[1][]{d_{\mathrm{path}}\if\relax\detokenize{#1}\relax\else^{#1}\fi}

\newcommand*{\len}[2]{L_{#1}(#2)}
\newcommand*{\lenmetres}[2]{\len{\Fd^{#1}}{#2}}
\newcommand*{\paths}[4]{P(#1,#2;#3;#4)}

\newcommand{\lmet}[1]{L_{\Fd}(#1)}
\newcommand{\lmett}[1]{L_{\wt{\Fd}}(#1)}
\newcommand{\lmetconf}[2]{L^{#1}_{\Fd}(#2)}
\newcommand{\lmettconf}[2]{L^{#1}_{\wt{\Fd}}(#2)}

\newcommand*{\double}{{\mathrm{dbl}}}
\newcommand*{\ddouble}{{d_{\double}}}
\newcommand*{\angledouble}{{\theta_{\double}}}

\newcommand{\bad}{{\mathrm{bad}}}

\title[The conformally covariant metric on non-simple CLEs]{Existence and uniqueness\\ of the conformally covariant geodesic metric on non-simple conformal loop ensemble gaskets}

\usepackage[foot]{amsaddr}
\author{Jason Miller}
\author{Yizheng Yuan}
\address{Department of Pure Mathematics and Mathematical Statistics, University of Cambridge}

\begin{document}

\date{\today}
\setcounter{tocdepth}{1}

\parindent 0 pt
\setlength{\parskip}{0.20cm plus1mm minus1mm}

\begin{abstract}
We construct the canonical geodesic metric on the gasket of conformal loop ensembles (CLE$_\kappa$) in the regime $\kappa \in (4,8)$ where the loops intersect themselves, each other, and the domain boundary. Previous work of the authors and V.\ Ambrosio showed that the subsequential limits associated with certain approximation procedures for such a metric exist and are non-trivial. In this work, we show that the limit exists by proving that there is at most one geodesic metric on the CLE$_\kappa$ gasket which satisfies certain properties. Further, we obtain that the limit is conformally covariant. This paper is the foundation of future work which show that the metric for $\kappa=6$ is the continuum scaling limit of the chemical distance metric for critical percolation in two dimensions. We further conjecture that for $\kappa \in (4,8)$, the geodesic CLE$_\kappa$ metric is the scaling limit of the chemical distance metric associated with discrete models that converge to CLE$_\kappa$.
\end{abstract}

\maketitle

\tableofcontents

\section{Introduction}
\label{sec:intro}

\subsection{Overview}
\label{subsec:overview}

The \emph{Schramm-Loewner evolution} ($\SLE$) and the \emph{conformal loop ensemble} ($\CLE$) are canonical models that describe conformally invariant random geometries in the plane. SLE was introduced by Schramm in \cite{s2000sle} and it has been conjectured to describe the scaling limits of interfaces that arise from many lattice models in two dimensions and such convergence results have now been proved in a number of cases \cite{s2001percolation,ss2009dgff,s2010ising,lsw2004lerw}. CLE was introduced in \cite{s2009cle,sw2012cle} and is conjectured to describe the joint scaling limit of all of the interfaces in such models and such convergence results have also been established in a number of cases \cite{bh2019ising,cn2006cle,ks2019fkising,lsw2004lerw}.  We remark that in the setting of random lattices a number of convergence results towards $\SLE$ and $\CLE$ have been proved \cite{s2016hc,kmsw2019bipolar,gm2021percolation,gm2021saw,gkmw2018active,lsw2017schnyder} using the framework developed in \cite{s2016zipper,dms2021mating}.

The $\SLE$ curves are indexed by a parameter $\kappa \geq 0$ ($\SLE_\kappa$) which determines the roughness of the curve.  For $\kappa = 0$, it is a smooth curve and as $\kappa$ increases it becomes increasingly fractal.  In particular, $\kappa \in (0,4]$ it is a simple curve, for $\kappa \in (4,8)$ it is self-intersecting but not space-filling, and for $\kappa \geq 8$ it is space filling \cite{rs2005basic}.  The $\CLE_\kappa$ are defined for $\kappa \in (8/3,8)$ and consist of a countable collection of loops each of which locally look like an $\SLE_\kappa$.  At the extreme $\kappa =8/3$ (resp.\ $\kappa=8$) it corresponds to the empty collection of loops (resp.\ a single space-filling loop).  For $\kappa \in (8/3,4]$, the loops are simple and do not intersect each other while for $\kappa \in (4,8)$ they are self-intersecting, intersect each other, and also intersect the domain boundary.  These are the same phases as for $\SLE_\kappa$ which were determined in \cite{rs2005basic}.  In this work, we will focus on the \emph{non-simple regime}, which is the regime that $\kappa \in (4,8)$.

The purpose of the present paper is show the existence and uniqueness of a \emph{geodesic metric} on the \clek{} gasket for each $\kappa \in (4,8)$. Concretely, we prove that there exists a unique (up to multiplicative constants) \emph{geodesic} metric on the \clek{} gasket satisfying the \emph{domain Markov property}. We further prove that this metric is conformally covariant. Previous work of the authors and V.\ Ambrosio \cite{amy2025tightness} initiated the program of constructing natural metrics in the gasket of a $\CLE_\kappa$ for $\kappa \in (4,8)$. In particular, it is shown in \cite{amy2025tightness} that certain approximations are tight and give rise to non-trivial geodesic metrics as subsequential limits. The present work builds on \cite{amy2025tightness} and shows that these approximations converge to the geodesic \clek{} metric. We conjecture that the geodesic \clek{} metric is the scaling limit of the chemical distance metric for discrete models in two dimensions that converge to $\CLE_\kappa$. In the case $\kappa = 6$, building on the result of this paper, we will show in future work \cite{dmmy2025percolation} that it is the scaling limit of the chemical distance metric for critical percolation on the triangular lattice.

\begin{figure}[ht]
\centering
\includegraphics[width=0.3\textwidth]{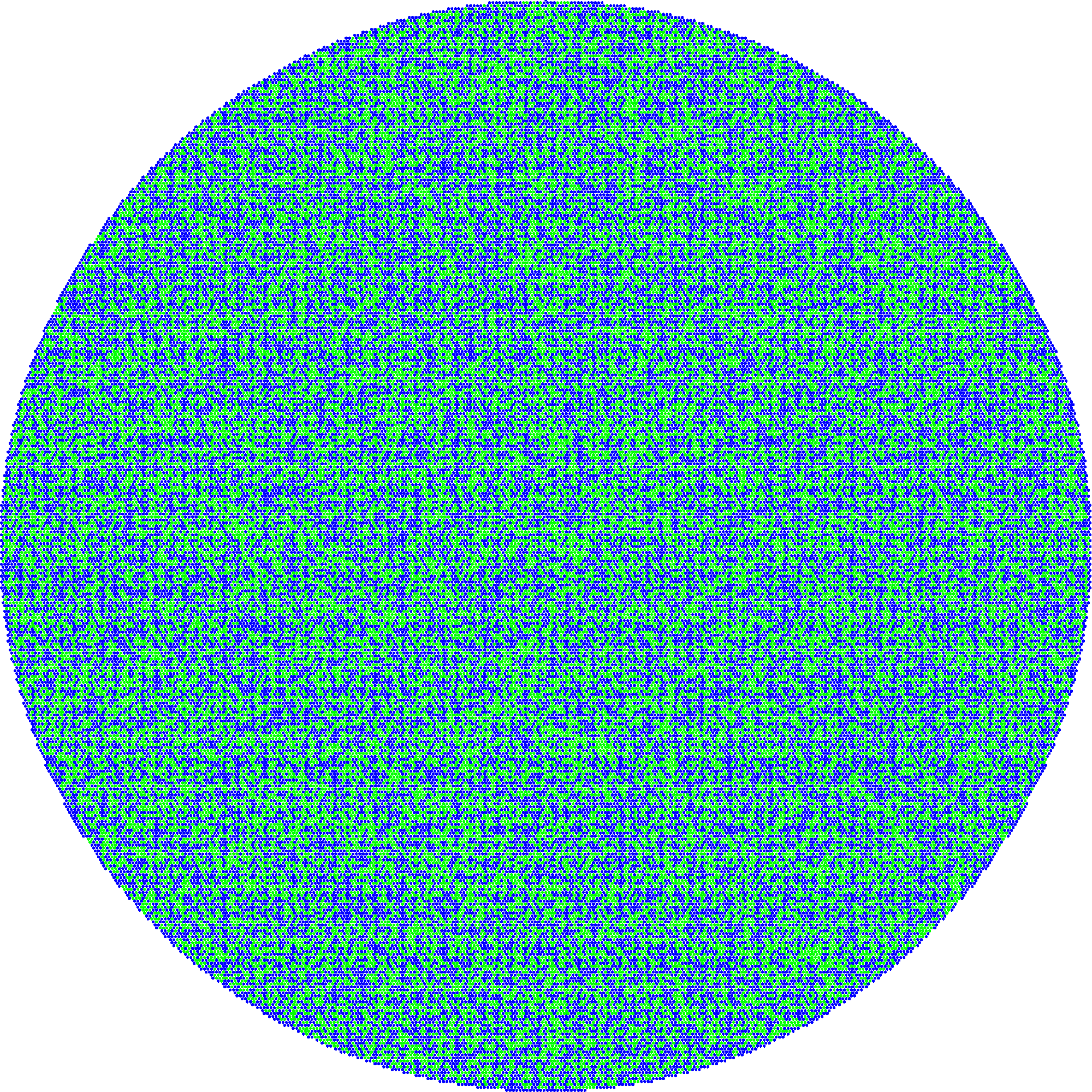} \hspace{0.01\textwidth} \includegraphics[width=0.3\textwidth]{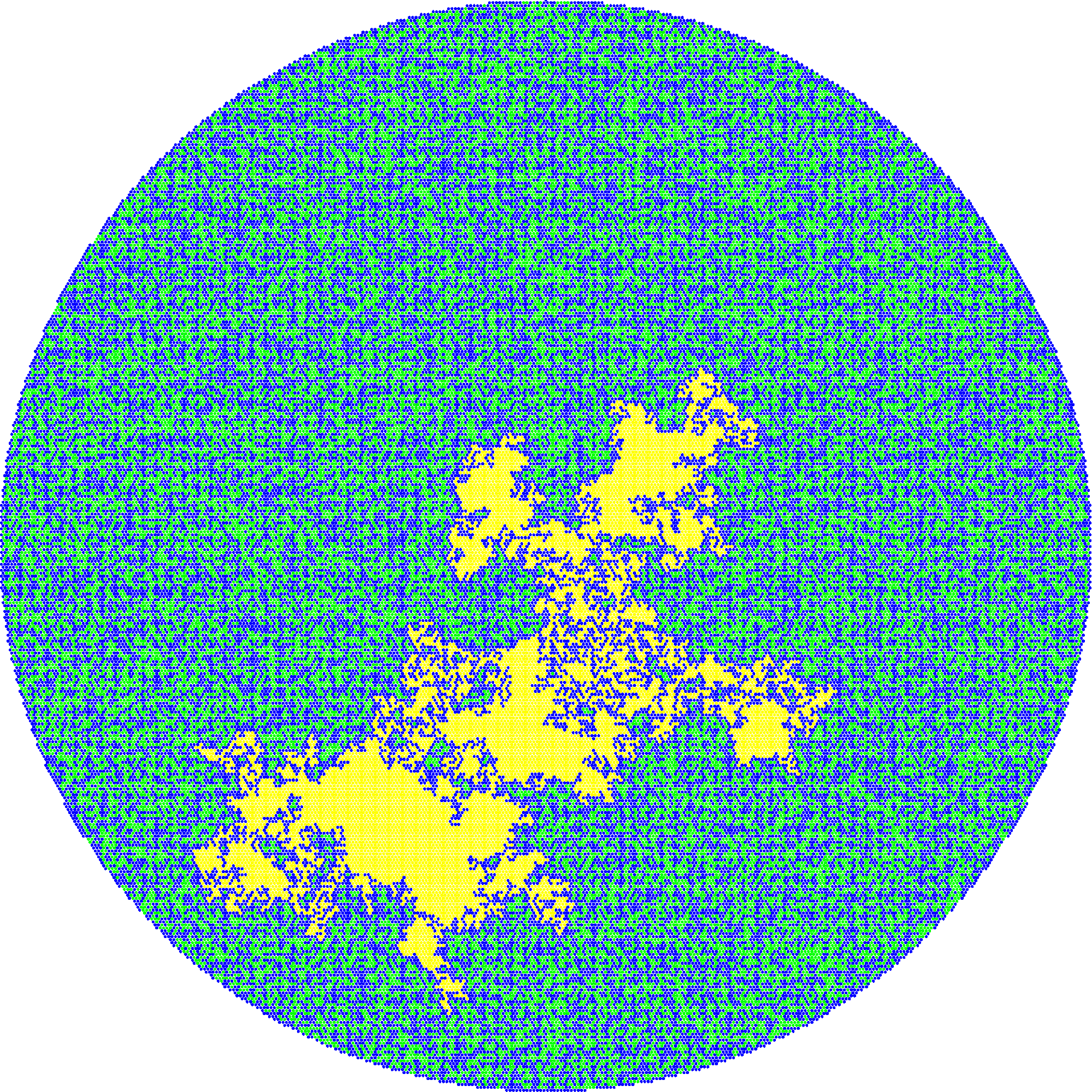} \hspace{0.01\textwidth} \includegraphics[width=0.3\textwidth]{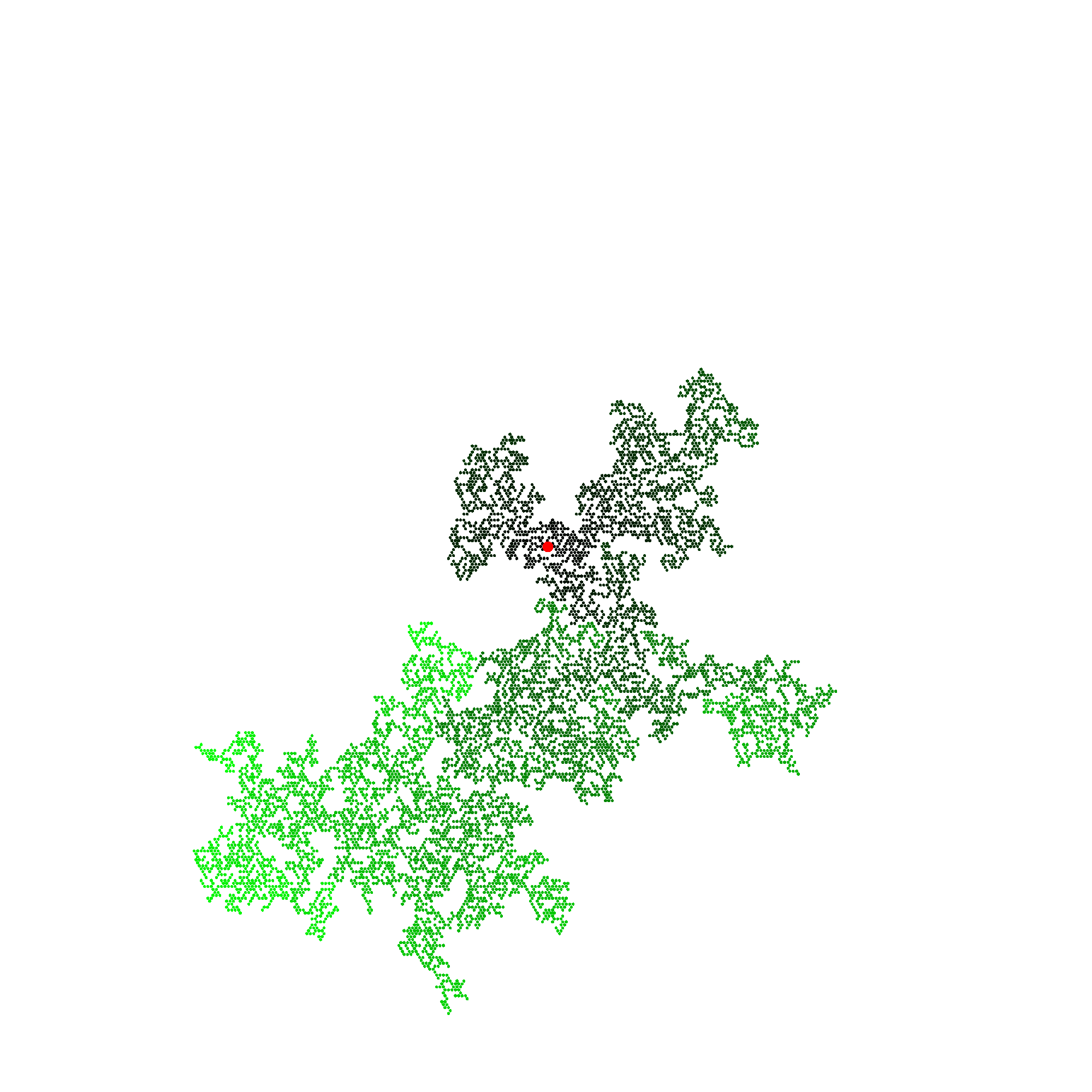}
\caption{{\bf Left:} A sample of critical percolation on the triangular lattice in the unit disk. {\bf Middle:} The regions surrounded by the outermost open cluster that surrounds the origin are shown in yellow. {\bf Right:} The corresponding cluster, colored according to their graph distance to a point near the center.}
\end{figure}

In the remainder of this article, we use $\kappa' \in (4,8)$ to denote the parameter for $\SLE$ and $\CLE$ in the non-simple regime and let $\kappa=16/\kappa' \in (2,4)$ be the dual value. We describe the setup for our main results in Section~\ref{se:setup_main}. The definition of a geodesic \clekp{} metric is given in Section~\ref{se:geodesic_metric}. We then state our main results in Section~\ref{se:main_results}.  There has been a substantial amount of previous work on chemical distance metrics in related settings which we will review in Section~\ref{subsec:other_work}.

\subsection{Setup}
\label{se:setup_main}

Before we state our main results, we need to describe the setup and give the definition of a geodesic \clekp{} metric. We will give two definitions of a geodesic \clekp{} metric. The definition of a \emph{weak} geodesic \clekp{} metric will be a special case of the class of metrics considered in \cite{amy2025tightness} and it has been proved in \cite{amy2025tightness} that the subsequential limits of geodesic approximation schemes satisfy the definition. We will then show that such a metric satisfies some stronger properties, and we will call them a \emph{(strong)} geodesic \clekp{} metric.

\begin{definition}
Let $\Gamma$ be a collection of loops in $\C$. We say that a path $\gamma$ is \emph{admissible for $\Gamma$} if it does not cross any loop of $\Gamma$.\footnote{This means that for any subsegment $\ell \subseteq \CL$ of a loop $\CL \in \Gamma$, if we let $\varphi$ be a conformal transformation from the unbounded connected component of $\C\setminus\ell$ to $\C\setminus\ol{\D}$, then $\varphi(\gamma \setminus \Fill(\ell))$ extends to a continuous path in $\C\setminus\D$.} For each $U \subseteq \C$, we let $\paths{x}{y}{U}{\Gamma}$ denote the set of admissible paths for $\Gamma$ from $x$ to $y$ within $U$.
\end{definition}

We need to describe the topology in which we view the \clekp{} gaskets, as considered in \cite{amy2025tightness}. Recall that in the regime $\kappa' \in (4,8)$ the loops of a \clekp{} intersect themselves, and hence there are points that can be accessed from two different sides of a loop. Therefore we do not view the \clekp{} gasket as merely a subset of the plane, but rather as a metric space with an embedding in the plane (distinguishing points that are accessible from different sides of a loop).

Suppose that $\CL$ is a random non-self-crossing loop in $\C$ (we will always be considering the case where $\CL$ is a particular \clekp{} loop), and let $C$ be the regions surrounded by $\CL$ (i.e.\ the union of connected components of $\C \setminus \CL$ with winding number $1$). Given $\CL$, let $\Gamma_C$ be a conditionally independent \clekp{} in each connected component of $C$, and let $\Gamma = \{\CL\} \cup \Gamma_C$. For $U \subseteq \C$ and $x,y \in U$, we let (with $\diamE$ denoting Euclidean diameter)
\begin{equation}\label{eq:dpath}
 \dpath[U](x,y) = \inf\{ \diamE(\gamma) : \gamma \in \paths{x}{y}{U}{\Gamma} \} .
\end{equation}
We write $\dpath(x,y) = \dpath[\ol{C}](x,y)$.

Let $\wt{\Upsilon}_\Gamma \subseteq C$ be the points that do not lie on and are not surrounded by any loop of $\Gamma_C$. Let $(\Upsilon_\Gamma, \dpath)$ be the metric space completion of $(\wt{\Upsilon}_\Gamma, \dpath)$, and equip it with the natural embedding map $\Pi\colon \Upsilon_\Gamma \to \C$. We call $\Upsilon_\Gamma$ the gasket of $\Gamma_C$, and when no confusion occurs, we identify its points with points on $\C$ via the embedding map $\Pi$.

For each open, simply connected $U \subseteq \C$, let $\Gamma_{U^*} \subseteq \Gamma_C$ be the collection of loops that are entirely contained in $U$, and let $U^* \subseteq U \cap C$ be the set of points that are not on or inside any loop of $\Gamma_C\setminus\Gamma_{U^*}$. We view $U^*$ as the metric space equipped with the metric
\[ d_{U^*}(x,y) = \inf\{ \diamE(\gamma) : \gamma \in \paths{x}{y}{U}{\Gamma\setminus\Gamma_{U^*}} \} . \]
Recall \cite[Lemma~3.2]{gmq2021sphere} that the conditional law of $\Gamma_{U^*}$ given $\Gamma \setminus \Gamma_{U^*}$ is that of an independent collection of \clekp{} in each connected component of $U^*$. The gasket $\Upsilon_{\Gamma_{U^*}}$ of $\Gamma_{U^*}$ is defined analogously as a metric space $(\Upsilon_{\Gamma_{U^*}}, \dpath[U])$ equipped with an embedding map $\Pi\colon \Upsilon_{\Gamma_{U^*}} \to \C$.

The Gromov-Hausdorff-function (GHf) metric between compact metric spaces equipped with continuous functions \cite[Appendix~A]{amy2025tightness} is defined as follows. Suppose that $(X_i,d_i)$ for $i=1,2$ are non-empty compact metric spaces, and $f_i\colon X_i \to \R^n$, $i=1,2$, are continuous functions.\footnote{The GHf metric can be defined more generally, but for the sake of simplicity we restrict to this case here.} Then
\begin{equation}\label{eq:ghf}
\dGHf( (X_1,d_1,f_1), (X_2,d_2,f_2) ) = \inf d_\infty( f_1 \circ \psi_1^{-1}, f_2 \circ \psi_2^{-1}) 
\end{equation}
where the infimum is over all metric spaces $(W,d_W)$ and isometric embeddings $\psi_i \colon X_i \to W$ for $i=1,2$, and
\[ 
d_\infty( f_1 \circ \psi_1^{-1}, f_2 \circ \psi_2^{-1}) = \inf\left\{ \delta > 0 : \parbox{.55\linewidth}{For each $i=1,2$ and $x \in X_i$ there is $x' \in X_{3-i}$ with $d_W(\psi_i(x),\psi_{3-i}(x')) \le \delta$ and $\abs{f_i(x)-f_{3-i}(x')} \le \delta$} \right\} .
\]
It is shown in \cite[Appendix~A]{amy2025tightness} that $\dGHf( (X_1,d_1,f_1), (X_2,d_2,f_2) ) = 0$ if and only if there is an isometry $\psi$ between $(X_1,d_1)$ and $(X_2,d_2)$ with $f_1 = f_2 \circ \psi$. It is also shown that a collection of (isometry classes of) non-empty compact metric spaces equipped with continuous functions is relatively compact in the GHf topology if the metric spaces are relatively compact in the GH topology and the functions are uniformly bounded and equicontinuous.

\subsection{The geodesic \clekp{} metric}
\label{se:geodesic_metric}

Suppose that $\met{\cdot}{\cdot}{\Gamma}$ is a random geodesic metric on $\Upsilon_\Gamma$ that is continuous with respect to $\dpath$. The metric is not required to be determined by $\Gamma$, and we view $(\Gamma,\met{\cdot}{\cdot}{\Gamma})$ as a random variable with values in the product space of loop ensembles and metrics on their gaskets. We recall the definition of a geodesic metric.

\textbf{Geodesic:} For each $x,y \in \Upsilon_\Gamma$ there is an admissible path $\gamma$ from $x$ to $y$ with $\lmet{\gamma} = \met{x}{y}{\Gamma}$ where $\lmet{\gamma}$ denotes the length of the path $\gamma$ under the metric $\met{\cdot}{\cdot}{\Gamma}$.

For each subset $U \subseteq \C$ we let
\begin{equation}\label{eq:internal_metric}
 \metres{U}{x}{y}{\Gamma} = \inf_{\gamma \in \paths{x}{y}{U}{\Gamma}} \lmet{\gamma} ,\quad x,y \in U \cap \Upsilon_\Gamma .
\end{equation}
We call it the \emph{internal metric within $U$}. We remark that in \cite{amy2025tightness} the internal metrics are only defined in regions $\ol{V}$ for $V \in \metregions$, but since we are considering geodesic metrics, it does not make any difference here.

We further suppose that the following properties hold.

\textbf{Markovian property:} Let $U \subseteq \C$ be open, simply connected. The conditional law of $\metres{U}{\cdot}{\cdot}{\Gamma}$ given $\Gamma\setminus\Gamma_{U^*}$ and $\metres{\C\setminus\ol{U}}{\cdot}{\cdot}{\Gamma}$ is almost surely measurable with respect to $U^*$ (recall from just above that we view $U^*$ as a metric space).

\textbf{Translation invariance:} Let $U \subseteq \C$ be open, simply connected. There exists a probability kernel $\mu^{U^*}$ such that for each $z \in \C$, the conditional law of $\metres{U+z}{\cdot}{\cdot}{\Gamma}$ given $(U+z)^*$ is $(T_z)_* \mu^{U^*}(\cdot -z)$ where $T_z \Fd(x,y) = \Fd(x-z,y-z)$ denotes translation by $z$.

\begin{definition}
\label{def:cle_metric}
 Suppose that $\CL$ is a random non-self-crossing loop in $\C$, and let $C$ be the regions surrounded by $\CL$. Given $\CL$, let $\Gamma_C$ be a conditionally independent \clekp{} in each connected component of $C$, and let $\Gamma = \{\CL\} \cup \Gamma_C$. We call a random metric $\met{\cdot}{\cdot}{\Gamma}$ coupled with~$\Gamma$ that satisfies the assumptions above a \emph{weak geodesic \clekp{} metric}.
\end{definition}

Note that this definition implies that the collection $(\metres{\ol{V}}{\cdot}{\cdot}{\Gamma})_{V \in \metregions}$ defines a \clekp{} metric in the sense of \cite[Definition~1.5]{amy2025tightness} ($\metregions$ is defined in \cite{amy2025tightness}). In particular, all the results from \cite{amy2025tightness} apply to $\met{\cdot}{\cdot}{\Gamma}$.

We now give the definition of a (strong) \emph{geodesic \clekp{} metric}. Suppose that $\met{\cdot}{\cdot}{\Gamma}$ is a geodesic metric on $\Upsilon_\Gamma$ that is continuous with respect to $\dpath$. (In contrast to the weak geodesic \clekp{} metric, the (strong) geodesic \clekp{} metric is determined by $\Gamma$ by definition.) We say that $\met{\cdot}{\cdot}{\Gamma}$ is a geodesic \clekp{} metric if it satisfies the following properties.

\textbf{Geodesic:} For each $x,y \in \Upsilon_\Gamma$ there is an admissible path $\gamma$ from $x$ to $y$ with $\lmet{\gamma} = \met{x}{y}{\Gamma}$ where $\lmet{\gamma}$ denotes the length of the path $\gamma$ under the metric $\met{\cdot}{\cdot}{\Gamma}$.

\textbf{Locally determined:} For each open, simply connected $U \subseteq \C$ there is a measurable function $F^U$ of $(U^*,\Gamma_{U^*})$ such that $\metres{U}{\cdot}{\cdot}{\Gamma} = F^U(U^*,\Gamma_{U^*})$ almost surely.

\textbf{Translation invariance:} For each open, simply connected $U \subseteq \C$ and $z \in \C$, the function $F^U$ defined above satisfies $F^{U+z} = T_z \circ F^U(\cdot -z)$ where $T_z \Fd(x,y) = \Fd(x-z,y-z)$ denotes translation by $z$.

\textbf{Scale covariance:} There exists a constant $\alpha > 0$ such that for each open, simply connected $U \subseteq \C$ and $\lambda > 0$, the function $F^U$ defined above satisfies $F^{\lambda U} = \lambda^\alpha S_\lambda \circ F^U(\cdot/\lambda)$ where $S_\lambda \Fd(x,y) = \Fd(\lambda^{-1}x,\lambda^{-1}y)$.

\begin{definition}\label{def:strong_cle_metric}
 Suppose that $\CL$ is a random non-self-crossing loop in $\C$, and let $C$ be the regions surrounded by $\CL$. Given $\CL$, let $\Gamma_C$ be a conditionally independent \clekp{} in each connected component of $C$, and let $\Gamma = \{\CL\} \cup \Gamma_C$. We call a random metric $\met{\cdot}{\cdot}{\Gamma}$ coupled with~$\Gamma$ that satisfies the assumptions above a \emph{geodesic \clekp{} metric}.
\end{definition}

\subsection{Main results}
\label{se:main_results}

We assume throughout that $\kappa' \in (4,8)$. In our main Theorems~\ref{thm:uniqueness}--\ref{thm:conformal_covariance} we consider the following setup. Let $\Gamma_\D$ be a nested $\CLE_{\kappa'}$ in $\D$, let $\CL$ be the outermost loop of $\Gamma_\D$ that surrounds~$0$, and let $C$ be the regions surrounded by $\CL$. Let $\Gamma_C$ be the loops of $\Gamma_\D$ contained in $\ol{C}$, and let $\Gamma = \{\CL\} \cup \Gamma_C$. Further, we \emph{condition on the event that $\CL \subseteq \D$} (i.e., $\CL \cap \partial \D = \varnothing$), and suppose that $\met{\cdot}{\cdot}{\Gamma}$ is a weak geodesic \clekp{} metric coupled with~$\Gamma$. (It is shown in \cite[Theorem~1.17]{amy2025tightness} that such a metric exists.)

We note that proving our main theorems in this concrete setup is enough, and the results generalize to all \emph{interior} clusters of a \clekp{} in a general simply connected domain, and of a whole-plane \clekp{}; this will be shown in the Theorems~\ref{th:whole_plane_metric}--\ref{th:metric_general_domain} below. We emphasize however the importance of considering an interior cluster. If we defined a geodesic metric on the gasket of a \clekp{} in a smooth domain, then the metric would degenerate along the smooth boundary since we expect the geodesics in the interior to have dimension larger than $1$, so that smooth paths would have zero length.

\begin{theorem}
\label{thm:uniqueness}
The following hold.
\begin{itemize}
\item Every weak geodesic \clekp{} metric is a geodesic \clekp{} metric. 
\item Suppose that $\met{\cdot}{\cdot}{\Gamma}$, $\mett{\cdot}{\cdot}{\Gamma}$ are geodesic $\CLE_{\kappa'}$ metrics.  There exists a constant $c > 0$ so that $\met{\cdot}{\cdot}{\Gamma} = c \mett{\cdot}{\cdot}{\Gamma}$.
\end{itemize}
\end{theorem}

As a consequence of Theorem~\ref{thm:uniqueness}, we can use the scaling property to define the \clekp{} metric in the interior clusters of the \clekp{} in any domain, and for the whole-plane \clekp{}. The following shows that the scaling exponent is unique. Let $d_\SLE$ be the dimension of an \slek{} curve \cite{rs2005basic,b2008dimension} and $\ddouble$ be the dimension of the double points of an \slekp{} curve \cite{mw2017intersections} (see~\eqref{eq:dsle}, \eqref{eq:ddouble} for their precise values).

\begin{theorem}
\label{thm:exponent}
There exists a constant $\alpha \in [\ddouble \vee 1, d_\SLE]$ depending only on $\kappa'$ such that every \clekp{} metric satisfies the scale covariance property with this exponent $\alpha$.
\end{theorem}

The following gives that this metric is conformally covariant with the same exponent $\alpha$ as in Theorem~\ref{thm:exponent}.

\begin{theorem}
\label{thm:conformal_covariance}
Suppose that $\met{\cdot}{\cdot}{\Gamma}$ is a geodesic \clekp{} metric. Suppose that $U, \wt{U} \subseteq \D$ are simply connected domains and $\varphi \colon U \to \wt{U}$ is a conformal map. For each admissible path $\gamma\colon [0,1] \to U \cap \Upsilon_\Gamma$, let $\lmet{\gamma}$ be its length with respect to $\metres{U}{\cdot}{\cdot}{\Gamma}$ and $\lmett{\varphi(\gamma)}$ the length of $\varphi(\gamma)$ with respect to $\metres{\wt{U}}{\cdot}{\cdot}{\varphi(\Gamma)}$.\footnote{It is explained in Section~\ref{sec:conformally_covariant} why the latter is well-defined.} Then we have
\[
 \lmett{\varphi(\gamma)} = \int_0^1 \abs{\varphi'(\gamma(t))}^\alpha \,d\lmet{\gamma}
\]
where $\alpha$ is the constant from Theorem~\ref{thm:exponent}.
\end{theorem}

Finally, we explain that the setup considered above can be extended to define a geodesic metric on all interior clusters of the \clekp{} in a general domain and of the whole-plane \clekp{}. We recall that a whole-plane (nested) $\CLE_{\kappa'}$ can be defined as an appropriate limit of (nested) $\CLE_{\kappa'}$ in any sequence of domains $D_n$ that tend to $\C$ in the sense that for every $R > 0$ there exists $n_0 \in \N$ such that $n \geq n_0$ implies that $B(0,R) \subseteq D_n$ \cite{mww2016extreme}.

Suppose that $D \subseteq \C$ is a simply connected domain or $D = \C$. Let $\Gamma^D$ be a nested \clekp{} in $D$ (resp.\ a whole-plane nested \clekp{}). We define the collection of its interior gaskets as the metric space completions with respect to $\dpath[D]$ of the set of points that do not lie on any loop of $\Gamma^D$. Note that in the case $D \subsetneq \C$, this includes the points on the exterior gasket (i.e.\ the points connected to $\partial D$), but the exterior gasket is split up by $\partial D$ into smaller gaskets which do not contain positive intervals of $\partial D$.

Let $\Upsilon_1,\Upsilon_2,\ldots$ be an enumeration of the interior gaskets of $\Gamma^D$ (we will also refer to them as clusters). We consider a metric $\met{\cdot}{\cdot}{\Gamma^D}$ as a countable collection of metrics defined on each cluster. For $U \subseteq \C$ open, we consider the internal metric $\metres{U}{\cdot}{\cdot}{\Gamma^D}$ defined as follows. For each $i \in \N$, if the cluster $\Upsilon_i$ intersects $U$, then $\Upsilon_i \cap U$ decomposes into at most countably many connected components (with respect to $\dpath[U]$). On each component, the internal metric is defined via~\eqref{eq:internal_metric}. Therefore we can view $\metres{U}{\cdot}{\cdot}{\Gamma^D}$ as a countable sequence of metric spaces, each equipped with an embedding into the plane. Further, we let $\Gamma^D\big|_U$ be the countable collection of loops and strands of $\CL \cap U$ for each $\CL \in \Gamma^D$. (The number of strands is at most countable because $\CL$ is a continuous path.)

\begin{theorem}\label{th:whole_plane_metric}
 Consider the setup described just above with $D = \C$. There exists a unique (up to a deterministic constant) metric $\met{\cdot}{\cdot}{\Gamma^\C}$, measurable with respect to $\Gamma^\C$, that is continuous with respect to $\dpath$ and satisfies the following properties:
 \begin{itemize}
  \item \textbf{Geodesic:} For each $i \in \N$ and $x,y \in \Upsilon_i$ there is a path $\gamma$ in $\Upsilon_i$ from $x$ to $y$ with $\lmet{\gamma} = \met{x}{y}{\Gamma^\C}$.
  \item \textbf{Locally determined:} For each open $U \subseteq \C$, the internal metric $\metres{U}{\cdot}{\cdot}{\Gamma^\C}$ is almost surely a measurable function of $\Gamma^\C\big|_U$.
  \item \textbf{Translation invariance:} For each $z \in \C$, we have $\met{\cdot+z}{\cdot+z}{\Gamma^\C+z} = \met{\cdot}{\cdot}{\Gamma^\C}$ almost surely.
 \end{itemize}
 Moreover, there is a constant $\alpha \in [\ddouble \vee 1, d_\SLE]$ depending only on $\kappa'$ such that the metric $\met{\cdot}{\cdot}{\Gamma^\C}$ satisfies
 \begin{itemize}
  \item \textbf{Scale covariance:} For each $\lambda > 0$ we have $\met{\lambda\cdot}{\lambda\cdot}{\lambda\Gamma^\C} = \lambda^\alpha\met{\cdot}{\cdot}{\Gamma^\C}$ almost surely.
  \item \textbf{Conformal covariance:} Suppose $U \subseteq \C$ is a simply connected domain and $\varphi\colon U \to \wt{U}$ is a conformal transformation. Then almost surely, for each cluster $\Upsilon_i$ of $\Gamma^\C$, each $x,y \in \Upsilon_i \cap U$, and each $\gamma \in \paths{x}{y}{U}{\Gamma^\C}$,
  \begin{align*}
   \lenmetres{\wt{U}}{\varphi(\gamma)} &= \int \abs{\varphi'(\gamma(t))}^\alpha \,d\lenmetres{U}{\gamma} \\
   \text{and}\quad \metres{\wt{U}}{\varphi(x)}{\varphi(y)}{\varphi(\Gamma^\C\big|_U)} &= \inf_{\gamma \in \paths{x}{y}{U}{\Gamma^\C}} \lenmetres{\wt{U}}{\varphi(\gamma)} .
  \end{align*}
 \end{itemize}
\end{theorem}

\begin{theorem}\label{th:metric_general_domain}
 Consider the setup described just above. There is a collection of metrics $\met{\cdot}{\cdot}{\Gamma^D}$ for each simply connected domain $D \subseteq \C$, each measurable with respect to $\Gamma^D$, respectively, that is continuous with respect to $\dpath$ and satisfies the following properties:
 \begin{itemize}
  \item \textbf{Geodesic:} For each $i \in \N$ and $x,y \in \Upsilon_i$ there is a path $\gamma$ in $\Upsilon_i$ from $x$ to $y$ with $\lmet{\gamma} = \met{x}{y}{\Gamma^D}$.
  \item \textbf{Locally determined:} For each open $U \subseteq \C$, the internal metric $\metres{U}{\cdot}{\cdot}{\Gamma^D}$ is almost surely a measurable function of $\Gamma^D\big|_U$ that does not depend on the choice of $D$.
  \item \textbf{Translation invariance:} For each simply connected domain $D \subseteq \C$ and $z \in \C$, the two metrics $\met{\cdot}{\cdot}{\Gamma^D}$ and $\met{\cdot}{\cdot}{\Gamma^{D+z}}$ satisfy $\met{\cdot+z}{\cdot+z}{\Gamma^D+z} = \met{\cdot}{\cdot}{\Gamma^D}$ almost surely.
 \end{itemize}
 Any two such collections of metrics differ by a global deterministic constant. Such a collection of metrics is conformally covariant in the following sense. There is a constant $\alpha \in [\ddouble \vee 1, d_\SLE]$ depending only on $\kappa'$ such that if $\varphi\colon D \to \wt{D}$ is a conformal transformation, then almost surely, for each admissible path $\gamma \subseteq D$,
 \[
  \len{\met{\cdot}{\cdot}{\varphi(\Gamma^D)}}{\varphi(\gamma)} = \int \abs{\varphi'(\gamma(t))}^\alpha \,d\len{\met{\cdot}{\cdot}{\Gamma^D}}{\gamma} .
 \]
\end{theorem}

As part of the proofs we also prove the following concentration estimates for the geodesic \clekp{} metric.

\begin{proposition}
Consider the setup described just above (where either $D = \C$ or $D \subseteq \C$). Let $\alpha$ be the constant in the theorems above. Fix $a>0$ and a compact set $K \subseteq D$. Then
\[
 \p\left[ \sup_{i \in \N} \sup_{\substack{x,y \in \Upsilon_i \cap K\\ \dpath(x,y) < \delta}} \metres{B(x,3\delta)}{x}{y}{\Gamma} \ge \delta^{\alpha-a} \right] = o^\infty(\delta) \quad\text{as } \delta \searrow 0 ,
\]
and
\[
 \p\left[ \inf_{i \in \N} \inf_{\substack{x,y \in \Upsilon_i \cap K\\ \dpath(x,y) \ge \delta}} \met{x}{y}{\Gamma} \le \delta^{\alpha+a} \right] = o^\infty(\delta) \quad\text{as } \delta \searrow 0 .
\]
\end{proposition}

The upper bound has already been shown in \cite[Corollary~5.19]{amy2025tightness}, and the present formulation follows by applying it to the rescaled metric $\met{\delta\cdot}{\delta\cdot}{\delta\Gamma}$ and using that $\median[\delta]{} = c\delta^\alpha$ (where $\median[\delta]{}$ is defined in~\eqref{eq:quantile_def}). The lower bound will follow from Lemma~\ref{le:len_lb_conditional} and the proof of Proposition~\ref{pr:bilipschitz}.

\subsection{Relationship to other work}
\label{subsec:other_work}

We will now review previous work on chemical distance metrics which is related to the present work.

\subsubsection{Supercritical percolation}
There have been a number of works that have studied chemical distance metrics in the context of the percolation model, especially in the case of \emph{supercritical percolation}.  In this regime, it turns out that the behavior of the metric is very strongly Euclidean.  Work in this direction started with the seminal contribution of Grimmett-Marstrand \cite{gm1990supercritical} and increasingly strong concentration bounds have been established by Antal-Pisztora \cite{ap1996chemical} and Garet-Marchand \cite{gm2007chemical}.  The strong Euclidean nature of the supercritical regime is also apparent in the results focused on the behavior of random walk.  In particular, Barlow proved Gaussian heat kernel estimates in \cite{bar-rw-supercritical} and the convergence of the random walk to standard Brownian motion was proved in the independent works \cite{bb-rw-percolation,mp-rw-percolation}.

\subsubsection{Critical percolation}

The results that have been established thus far in the case of critical percolation are much weaker and less is known in this regime. The question whether SLE based techniques can be used to study the chemical distance metric on critical percolation was mentioned in Schramm's ICM contribution \cite[Problem~3.3]{s2007icm}. The present paper is the first that achieves this.

Previous works by Aizenman-Burchard \cite{ab1999holder} have shown that the length of the shortest open crossing of an $n \times n$ box grows faster than $n^\beta$ for some exponent $\beta > 1$, with probability converging to $1$. It was shown by Damron-Hanson-Sosoe \cite{dhs2017chemical,dhs2021chemicalineq} that the length of the shortest open crossing (on the event that a crossing exists) grows with an exponent that is slower than the length of the lowest crossing (which is $4/3$ for site percolation on the triangular lattice), verifying a conjecture of Kesten-Zhang \cite{kz1993conj}. Radial versions of these statements are proved in \cite{sr2022radialchemical}. Numerical simulations \cite{hs1988percnumerical,dzgss2010simulations} suggest that the length should grow with an exponent that is approximately $1.13$.

The relation to random walk on the percolation cluster has been studied in the work \cite{gl2022chemicalsub} which shows its subdiffusivity with respect to the chemical distance.

For percolation in high dimensions, which behaves very differently, there have also been some works studying the chemical distance metric. See \cite{cchs-chemical-high-dim,brbn-hypercube-percolation} (in the latter it is also announced that the metric space scaling limit of the high dimensional clusters will be proved).

\subsubsection{Long-range percolation}

There has also been a substantial amount of work on the metric in the case of long-range percolation, which is defined by setting edges between any pair of vertices $x,y \in \Z^d$ to be open independently with probability $\| x - y\|^{-s}$ where $s > 0$ is a parameter.  The metric associated with this model has been investigated in works of Biskup \cite{b2004chemical,b2011chemical}, Ding-Sly \cite{ds2013longrange}, B\"aumler \cite{b2023chemical}, and \cite{dfh2023chemicallongrange}.

\subsubsection{Uniform spanning tree}

In the case of the uniform spanning tree model in two dimensions, the tightness of the chemical distance metric was proved in \cite{bm2010lerw,bm2011ust}.  The continuum limit can be thought of as the $\CLE_8$ metric and this was constructed in \cite{hs2018euclideanmating}.  In this case, the metric has a very different character as it is a random tree and the geodesics consist of the $\SLE_2$-type curves which define the branches and the length of the branch is defined by the natural parameterization of $\SLE$ developed in \cite{ls2011natural}.  (See also \cite{lv2021natural} for the convergence of LERW to $\SLE_2$ in the natural parameterization.)

\subsubsection{Liouville quantum gravity metric}

Recall that a Liouville quantum gravity (LQG) surface is formally described by the expression
\begin{equation}
\label{eqn:lqg_expression}
 e^{\gamma h(z)} (dx^2 + dy^2)
\end{equation}
where $\gamma \in (0,2]$ is a parameter and $h$ is an instance of (some form of) the Gaussian free field (GFF) on a planar domain $D$.  Since the GFF is a distribution and not a function, it is a non-trivial to make mathematical sense of~\eqref{eqn:lqg_expression}.  The metric associated with~\eqref{eqn:lqg_expression} (the two-point distance function) was first constructed in the case that $\gamma=\sqrt{8/3}$ in \cite{ms2020lqgtbm1,ms2021lqgtbm2,ms2021lqgtbm3} using the $\SLE$-based techniques developed in \cite{s2016zipper,dms2021mating}.  The metric was subsequently constructed for $\gamma \in (0,2)$ using an approach based on first regularizing the GFF $h$ to define approximations for the metric and establishing tightness \cite{dddf2020lqgmetric} and then proving the existence and uniqueness of the limit in \cite{gm2021lqgmetric} by showing that it is uniquely characterized by the list of axioms which were laid out in \cite{mq2020notsle}. The LQG metric in the so-called \emph{supercritical regime} was recently constructed in \cite{dg2023supertightness, dg2023superuniqueness} following a similar strategy. The present work together with \cite{amy2025tightness} is related to \cite{dddf2020lqgmetric, gm2021lqgmetric} but the techniques that we use are very different.

\subsubsection{Random planar maps}

It is shown in \cite{lgm2019largefaces,cmr2023largefaces} that a class of random planar maps with heavy-tailed face degree distributions converge in the scaling limit as metric spaces. Their limits admit explicit descriptions as the $\alpha$-stable gaskets (resp.\ carpets) where $\alpha \in (1,2)$ is a parameter. There are two different regimes of the parameter in this model. The regime $\alpha \in (1,3/2)$ should correspond to \clekp{} on LQG with $\kappa' \in (4,8)$, and the regime $\alpha \in [3/2,2)$ should correspond to \clek{} on LQG with $\kappa \in (8/3,4]$.

\subsubsection{Further results}

Let us finally mention that there are a number of further models where chemical distance metrics have been investigated, but which are less related to the present work.  This includes the model of random interlacements \cite{cp2012interlacements, hpr2023interlacements} and GFF percolation and random walk loop soups \cite{dl2018gffrwk}.

\subsection{Outline}
\label{subsec:outline}

We will now give the outline of the proof of our main theorems. The intuition is that the CLE in a small neighborhood around each point is self-similar in law, and its law approximately the same around each point $z \in \D$. By the Markovian property and the translation invariance of the metric, we can think of it as being i.i.d.\ in small neighborhoods of each point. There is a law of large numbers phenomenon emerging, so that the length of geodesics is approximated by their mean in each of the small neighborhoods that they pass. Making this intuition rigorous, however, requires highly non-trivial techniques. We will make use of the independence across scales technology developed in \cite{amy-cle-resampling}, along with the strong tail estimates for the metric proved in \cite{amy2025tightness}.

In the first step we suppose that we have two weak geodesic \clekp{} metrics $\met{\cdot}{\cdot}{\Gamma}$ and $\mett{\cdot}{\cdot}{\Gamma}$ and aim to show that $\met{\cdot}{\cdot}{\Gamma} = c \mett{\cdot}{\cdot}{\Gamma}$ for a deterministic constant $c$. We first suppose that the typical distances for $\met{\cdot}{\cdot}{\Gamma}$ and $\mett{\cdot}{\cdot}{\Gamma}$ across a set of Euclidean size $\delta$ are of comparable order as $\delta \searrow 0$. Under this hypothesis we show in Section~\ref{sec:bi_lip} that there exist deterministic constants $c_1, c_2$ such that
\begin{equation}
\label{eqn:bi_lip_equivalence}
 c_1 \met{\cdot}{\cdot}{\Gamma}  \leq \mett{\cdot}{\cdot}{\Gamma} \leq c_2 \met{\cdot}{\cdot}{\Gamma}.
\end{equation}
The idea is to cover space by small balls in which the two metrics are comparable, and to show that the probability that the space can be covered by such balls tends to $1$ as the size of the balls go to $0$. By the independence across annuli developed in \cite{amy-cle-resampling}, we can find with overwhelming probability a positive fraction of annuli in which the two metrics have typical behavior.  A significant difficulty which is not present in other settings is that the topology of the \clekp{} gasket may force geodesics pass through the interior ball of an annulus, so that the lengths across the interior ball need to be controlled as well. This problem is addressed in Section~\ref{se:ball_crossing} where we show that with overwhelming probability we find scales on which the \clekp{} metric has typical behavior in the annulus and all subsequent annuli do not add to much length to the geodesics. We use the superpolynomial tail estimates proved in \cite{amy2025tightness} to make up for the lack of independence.

We then show in Section~\ref{sec:metrics_equivalent} that $c_1 = c_2$. We suppose that $c_1 < c_2$ are such that $c_1$ (resp.\ $c_2$) is the largest (resp.\ smallest) deterministic constant so that~\eqref{eqn:bi_lip_equivalence} holds and aim to arrive at a contradiction by showing that they can be improved. To this end, we argue that there must exist regions where $c_1$ (resp.\ $c_2$) are almost attained with positive probability. Then we use a spatial independence argument to show that with overwhelming probability such regions are all over the place, and significant portions of the geodesics pass through them. This means that the ratio of their lengths with respect to $\met{\cdot}{\cdot}{\Gamma}$ and $\mett{\cdot}{\cdot}{\Gamma}$ have to be strictly away from the best-case (resp.\ worst-case) constant $c_1$ (resp.\ $c_2$). This gives the desired contradiction and completes the proof of the uniqueness of the metric. This also shows that the metric is locally determined by applying the uniqueness result to two conditionally independent copies of $\met{\cdot}{\cdot}{\Gamma}$.

Once we have uniqueness, we immediately conclude the scale-covariance and the existence of the exponent $\alpha$. We next argue that the exponent $\alpha$ is the same for every geodesic \clekp{} metric. This also lifts the extra hypothesis on the comparability of the typical distances for $\met{\cdot}{\cdot}{\Gamma}$ and $\mett{\cdot}{\cdot}{\Gamma}$ across a set of Euclidean size $\delta$ which we have imposed earlier. Suppose that the exponents for $\met{\cdot}{\cdot}{\Gamma}$ and $\mett{\cdot}{\cdot}{\Gamma}$ are different and $\alpha' > \alpha$. Then on small scales $\mett{\cdot}{\cdot}{\Gamma}$ is typically much smaller than $\met{\cdot}{\cdot}{\Gamma}$. Using a covering argument similarly to Sections~\ref{sec:bi_lip} and~\ref{sec:metrics_equivalent}, we then conclude that in fact $\mett{\cdot}{\cdot}{\Gamma} = 0$. This is carried out in Section~\ref{se:exponent}.

In Section~\ref{se:general_domains} we extend the definition of the geodesic \clekp{} metric to the interior clusters in a general simply connected domain and in the whole-plane, and prove their conformal covariance. We prove in Section~\ref{sec:conformally_covariant} the conformal covariance in the setup of a single cluster as described at the beginning of Section~\ref{se:main_results}. The proof follows a similar strategy as in the Sections~\ref{sec:bi_lip} and~\ref{sec:metrics_equivalent}. We aim to cover the space with small balls in which the distances associated with $\varphi(\Gamma)$ and $\abs{\varphi'(z)}\Gamma$ are comparable up to a factor close to $1$. The challenge here is that the function that associates to $\Gamma$ the metric $\met{\cdot}{\cdot}{\Gamma}$ is not necessarily continuous, hence might not be stable under a slight conformal perturbation. We will utilize the local continuity in total variation of the (multichordal) \clekp{} law with respect to conformal perturbations proved in \cite{amy-cle-resampling}.

The extension to general domains is carried out in Section~\ref{se:whole_plane_metric} where we prove the Theorems~\ref{th:whole_plane_metric} and~\ref{th:metric_general_domain} using our main Theorems~\ref{thm:uniqueness}--\ref{thm:conformal_covariance}.

\subsection*{Acknowledgements}
J.M. and Y.Y.\ were supported by ERC starting grant SPRS (804116) and from ERC consolidator grant ARPF (Horizon Europe UKRI G120614). Y.Y.\ in addition received support from the Royal Society.

\section{Preliminaries}

Throughout the paper, we fix
\[ \kappa' \in (4,8) ,\quad \kappa = \frac{16}{\kappa'} \in (2,4).\]
We also set
\begin{align}
d_\SLE &= 1+\frac{\kappa}{8} = 1+\frac{2}{\kappa'} \quad\text{(dimension of $\SLE_\kappa$ \cite{rs2005basic,b2008dimension})},\label{eq:dsle}\\
\ddouble &= 2 - \frac{(12-\kappa')(4+\kappa')}{8\kappa'} \quad\text{(double point dimension of $\SLE_{\kappa'}$ \cite{mw2017intersections})},\label{eq:ddouble}\\
\angledouble &= \frac{\pi (\kappa-2)}{2-\kappa/2} \quad\text{(double point angle gap \cite{ms2016ig1,mw2017intersections})}.\label{eq:angledouble}
\end{align}

We will consider nested \clekp{} throughout this paper (even though the metric is defined only on a single cluster). This allow us to use the resampling tools from \cite{amy-cle-resampling} which we recall in Section~\ref{se:mcle}.

For $z\in\C$ and $r_2>r_1>0$, we denote $A(z,r_1,r_2) = B(z,r_2)\setminus\ol{B}(z,r_1)$. Let $\bIn A(z,r_1,r_2) = \partial B(z,r_1)$ and $\bOut A(z,r_1,r_2) = \partial B(z,r_2)$. We call a path $\gamma$ a crossing of $A(z,r_1,r_2)$ if it has one endpoint on each of the two boundary components of $\partial A(z,r_1,r_2)$ and is contained in $A(z,r_1,r_2)$ otherwise. We say that a path $\gamma$ crosses an annulus if it contains a crossing of the annulus.

We write $U \Subset V$ to denote that $U$ is compactly contained in $V$, i.e.\ $\ol{U}$ is compact and $\ol{U} \subseteq V$. We sometimes write $\distE$ (resp.\ $\diamE$) to denote Euclidean distance (resp.\ diameter).

For a function $f$ depending on $\delta \in (0,1]$, we write $f(\delta) = O(\delta^b)$ if $f$ is bounded by a constant times $\delta^b$.  We also write $f(\delta) = o^\infty(\delta)$ if $f(\delta) = O(\delta^b)$ for any fixed $b>0$. We write $a \lesssim b$ if $a \le cb$ for a constant $c>0$, and we write $a \asymp b$ if $a \lesssim b$ and $b \lesssim a$.

\subsection{Geodesic metric spaces}

We record the following lemma about internal metrics in a geodesic metric space.

\begin{lemma}\label{le:internal_metrics_compatibility}
Suppose that $(X,\Fd)$ is a geodesic metric space, and let $\len{\Fd}{\gamma}$ denote the length of a path $\gamma$. For $U \subseteq X$, define
\[ \Fd^U(x,y) = \inf_{\gamma' \subseteq U} \len{\Fd}{\gamma'} ,\quad x,y \in U, \]
where the infimum is taken over all paths $\gamma' \subseteq U$ connecting $x,y$. 
Then for each path $\gamma \subseteq U$ we have
\[ \len{\Fd^U}{\gamma} = \len{\Fd}{\gamma} . \]
\end{lemma}

\begin{proof}
Clearly, we have $\Fd \le \Fd^U$ due to the geodesic property. Therefore we only need to argue that $\len{\Fd^U}{\gamma} \le \len{\Fd}{\gamma}$ for each path $\gamma \subseteq U$. Suppose that $\gamma$ is parameterized on the time interval $[0,1]$, and let $(t_i)_{i=0,1,\ldots,n}$ be any finite partition of $[0,1]$. For each $i$ we have by definition $\Fd^U(\gamma(t_{i}),\gamma(t_{i+1})) \le \len{\Fd}{\gamma\big|_{[t_i,t_{i+1}]}}$, and hence
\[ \sum_i \Fd^U(\gamma(t_{i}),\gamma(t_{i+1})) \le \sum_i \len{\Fd}{\gamma\big|_{[t_i,t_{i+1}]}} = \len{\Fd}{\gamma} . \]
This shows that $\len{\Fd^U}{\gamma} \le \len{\Fd}{\gamma}$.
\end{proof}

\subsection{Schramm-Loewner evolution}

Let $D \subsetneq \C$ be a simply connected domain and $x,y$ two distinct prime ends of $D$. For each $\kappa > 0$, the \emph{chordal \slek{}} is a probability measure on non-crossing curves in $\ol{D}$ between $x$ and $y$. The family of \slek{} laws is conformally invariant in the sense that if $\varphi\colon D \to \wt{D}$ is a conformal transformation, then the pushforward of the \slek{} in $(D,x,y)$ under $\varphi$ is the \slek{} in $(\wt{D},\varphi(x),\varphi(y))$. On $(\h,0,\infty)$, the law is characterized by the \emph{chordal Loewner differential equation} as follows. For each $t \ge 0$, let $\h_t$ be the unbounded connected component of $\h \setminus \eta[0,t]$ and $g_t\colon \h_t \to \h$ be the conformal transformation with $g_t(z) = z+o(1)$ as $z \to \infty$. Then there exists a parameterization such that the family $(g_t(z))$ satisfies
\begin{equation}
\label{eqn:loewner_ode}
\partial g_t(z) = \frac{2}{g_t(z) - U_t} ,\quad z \in \h_t, \quad g_0(z) = z.
\end{equation}
with $U_t = \sqrt{\kappa}B_t$ and $(B_t)$ is a standard Brownian motion. The existence of the \slek{} curve was proved in \cite{rs2005basic,lsw2004lerw} (see also \cite{am2022sle8} for a continuum proof for $\kappa=8$). It is also proved in \cite{rs2005basic} that $\SLE_\kappa$ curves are simple for $\kappa \in (0,4]$, self-intersecting but not space-filling for $\kappa \in (4,8)$, and space-filling for $\kappa \geq 8$ \cite{rs2005basic}.

The \emph{$\SLE_\kappa(\rho)$ processes} are an important variant of $\SLE_\kappa$ where one keeps track of extra marked points \cite[Section~8.3]{lsw2003restriction}, \cite{sw2005coordinate}. Given a finite collection of \emph{force points} $(z^i)$, $z^i \in \ol{\h}$, and \emph{weights} $(\rho_i)$, $\rho_i \in \R$, the \slekr{\rho_1,\rho_2,\ldots} process in $(\h,0,\infty;z^1,z^2,\ldots)$ is given by solving~\eqref{eqn:loewner_ode} with the solution to
\begin{equation}
\label{eqn:sle_kappa_rho}
dW_t = \sqrt{\kappa} dB_t + \sum_i \re\left(\frac{\rho_i}{W_t - V_t^i}\right) dt, \quad dV_t^i = \frac{2}{V_t^i - W_t} dt,\quad V_0^i = z^i
\end{equation}
used in place of $(U_t)$ in~\eqref{eqn:loewner_ode}. It is shown in \cite{ms2016ig1} that~\eqref{eqn:sle_kappa_rho} has a unique solution up until the \emph{continuation threshold}, which is the first time $t$ that the sum of the weights of the force points that have collided with $W_t$ is at most $-2$. Further, the resulting process is generated by a continuous curve. Again, the \slekr{\rho} process in a general simply connected domain is defined by mapping to~$\h$ via a conformal transformation.

\subsubsection{Bichordal \slekp{}}
\label{se:bichordal_sle}

\begin{proposition}\label{pr:bichordal_sle}
For each $\kappa' \in (4,8)$ there is a unique law on pairs of curves $(\eta'_1,\eta'_2)$ such that almost surely there is a unique connected component of $\h \setminus \eta'_1$ (resp.\ $\h \setminus \eta'_2$) to the right of $\eta'_1$ (resp.\ left of $\eta'_2$) with $0,\infty$ on its boundary and the conditional law of $\eta'_2$ given $\eta'_1$ (resp.\ $\eta'_1$ given $\eta'_2$) is that of an \slekp{} in that component.
\end{proposition}

\begin{proof}
 The proof follows from the same argument as \cite[Theorem~4.1]{ms2016ig2} applied to the outer boundaries of $\eta'_1,\eta'_2$.
\end{proof}

The pair of curves $(\eta'_1,\eta'_2)$ in Proposition~\ref{pr:bichordal_sle} can be constructed as follows. Let $h$ be a GFF in $\h$ with boundary values $2\lambda'-\pi\chi = (\kappa-2)\lambda$ (resp.\ $-2\lambda'+\pi\chi = -(\kappa-2)\lambda$) on $\R_+$ (resp.\ $\R_-$). Let $\eta'_1$ (resp.\ $\eta'_2$) be the counterflow line of $h$ with angle $\theta' = \lambda'/\chi$ (resp.\ $-\theta' = -\lambda'/\chi$) from $\infty$ to $0$. Then the law of the pair $(\eta'_1,\eta'_2)$ is exactly the bichordal SLE$_{\kappa'}$. The flow line $\eta_1$ with angle $\theta = (\kappa-2)\lambda/(2\chi)$ from $0$ to $\infty$ agrees with the right boundary of $\eta'_1$, and the flow line $\eta_2$ with angle $-\theta = -(\kappa-2)\lambda/(2\chi)$ from $0$ to $\infty$ agrees with the left boundary of $\eta'_2$. Finally, note that the angle difference between $\eta_1,\eta_2$ is exactly the double point angle for SLE$_{\kappa'}$~\eqref{eq:angledouble}.

\subsection{Conformal loop ensembles}

Let $D \subsetneq \C$ be a simply connected domain. For each $\kappa \in (8/3,8)$, the \emph{conformal loop ensemble} \clek{} is a probability measure on countable collections of non-crossing loops in $\ol{D}$, each of which locally looks like an \slek{} \cite{sw2012cle,s2009cle}. The family of \clek{} laws is conformally invariant in the sense that if $\varphi\colon D \to \wt{D}$ is a conformal transformation, then the pushforward of the \clek{} in $D$ under $\varphi$ is the \clek{} in $\wt{D}$. This was proved for $\kappa \in (8/3,4]$ in \cite{sw2012cle} as a consequence of the conformal invariance of the Brownian loop soup and it was explained in \cite{s2009cle} that the conformal invariance of $\CLE_\kappa$ for $\kappa \in (4,8)$ follows from the reversibility of $\SLE_\kappa$ which was later proved in \cite{ms2016ig3}. Further, the \clek{} possess the following \emph{domain Markov property} \cite{sw2012cle,gmq2021sphere}. Suppose $K \subseteq \ol{D}$ is a closed connected set intersecting $\partial D$. Let $\Gamma$ be a \clek{} in $D$ and let $\Gamma(K) \subseteq \Gamma$ be the loops that intersect $K$. Then given $\Gamma(K)$, the conditional law of $\Gamma \setminus \Gamma(K)$ is given by an independent \clek{} in each connected component of $D \setminus \ol{\bigcup \Gamma(K)}$.

There are two variants of \clek{}, the \emph{nested} and the \emph{non-nested} \clek{}. The non-nested \clek{} is obtained as the collection of outermost loops of a nested \clek{} (i.e.\ the loops that are not surrounded by any of the other loops). Conversely, one obtains a nested \clek{} from a non-nested \clek{} by repeatedly sampling conditionally independent non-nested \clek{} in each of the remaining complementary connected components. The nested \clek{} can also be defined on the whole plane $\C$ \cite{gmq2021sphere}. It is characterized by the facts that it is locally finite and satisfies the domain Markov property described above for each closed set $K \subseteq \wh{\C}$ containing $\infty$.

We review the construction of \clekp{} in the case $\kappa' \in (4,8)$ \cite{s2009cle}. Fix a prime end $x \in \partial D$. The \emph{exploration tree} of a \clekp{} rooted at $x$ is defined as follows. Fix a countable dense set $(w_n)$ in $D$. For each~$n$ we let~$\eta'_n$ be an $\SLE_{\kappa'}({\kappa'}-6)$ process in $D$ from~$x$ to~$w_n$ where the force point is located infinitesimally to the right of $x$ on $\partial D$.  We assume that the $\eta'_n$ are coupled together so as to agree until they separate their target points \cite{sw2005coordinate} and then evolve independently afterwards. For each $z \in D$ we define a loop $\CL_z$ surrounding $z$ as follows. Consider any branch $\eta'_n$ of the exploration tree that surrounds $z$ clockwise. Let $\tau$ be the first time when it surrounds $z$ clockwise, and let $\sigma < \tau$ be the last time when the driving function of $\eta'_n$ has collided with the force point process. Let $\CL_z$ be the loop formed by the segment $\eta'_n[\sigma,\tau]$ followed by concatenating the branches that target the component adjacent to $\eta'_n(\sigma)$. It is proved in \cite{ms2016ig3} that the law of the resulting loop ensemble does not depend on the choice of~$x$ (i.e., is root invariant). This yields the non-nested \clekp{} in $D$. The procedure can be repeated inside the loops (we can use the same exploration tree) to construct the nested \clekp{} in $D$. By the construction, the exploration tree is determined by $\Gamma$ and the root $x$.

It is proved in \cite{ms2017ig4} that the \clek{} is \emph{locally finite}, meaning that for each $\epsilon > 0$ the number of loops with diameter at least $\epsilon$ is finite. This can be made quantitative \cite{ghm2020kpz}.

\begin{lemma}\label{le:loop_crossings_tail}
Let $\Gamma$ be a \clekp{} in $D$. Let $A$ be an annulus contained in $D$, and let $N_A$ be the number of crossings of $A$ made by loops of $\Gamma$. Then $\p[N_A > M] = o^\infty(M^{-1})$ as $M \to \infty$.
\end{lemma}

\subsection{Multichordal \clekp{}}
\label{se:mcle}

Suppose $D$ is a simply connected domain, $\ul{x}$ is a finite even collection of distinct prime ends, and $\beta$ is an exterior planar link pattern between the marked points. The \emph{multichordal \clekp{}} in $(D;\ul{x};\beta)$ is a probability measure on the collections of loops and paths in $\ol{D}$ where each path connects a pair of marked points. It arises as the conditional law of the remainder given a \emph{partial exploration} of a \clekp{} (see Theorem~\ref{thm:cle_partially_explored} below). The \emph{interior link pattern} $\alpha$ between the marked points induced by the paths of a multichordal \clekp{} is random (its law depends on the conformal class of $(D;\ul{x};\beta)$), the conditional law of the paths given $\alpha$ is that of a \emph{global multiple \slekp{}}, and the conditional law of the loops given the paths is that of a \clekp{} in each complementary connected component. We refer to \cite{amy-cle-resampling} for the definition of the multichordal \clekp{} law, and will only summarize its main properties here. (See \cite[Theorem~1.6]{amy-cle-resampling}.)

\begin{proposition}
 The multichordal \clekp{} law is conformally invariant in the following sense. If $\Gamma$ is a multichordal \clekp{} in $(D;\ul{x};\beta)$ and $\varphi\colon D \to \wt{D}$ is a conformal transformation, then $\varphi(\Gamma)$ is a multichordal \clekp{} in $(\wt{D};\varphi(\ul{x});\beta)$.
\end{proposition}

\begin{proposition}\label{pr:link_probability}
 For each \emph{interior link pattern} $\alpha$, the probability that the chords of a multichordal \clekp{} in $(\D;\ul{x};\beta)$ induce the link pattern $\alpha$ is positive and depends continuously on $\ul{x}$.
\end{proposition}

We now define \emph{partial explorations} of a \clekp{}. Let $D \subseteq \C$ be a simply connected domain and $\varphi \colon D \to \D$ be a conformal transformation.  Let $\domainpair{D}$ consist of all pairs $(U,V)$ of simply connected sub-domains $U \subseteq V \subseteq D$ with $\dist(\varphi(D\setminus V), \varphi(U)) > 0$.  (Note that $\domainpair{D}$ does not depend on the choice of $\varphi$.)

Suppose $\Gamma$ is a nested \clekp{} in $D$, and $(U,V) \in \domainpair{D}$. The \emph{partial exploration} of $\Gamma$ in $D\setminus V$ until hitting $\ol{U}$, denoted by $\Gamma_\outside^{*,V,U}$, is the collection of maximal segments of loops and strands in $\Gamma$ that intersect $D \setminus V$ and are disjoint from $U$. Let $V^{*,U}$ be the connected component containing $U$ after removing from $D$ all loops and strands of $\Gamma_\outside^{*,V,U}$. This defines a simply connected domain with a finite number of marked points $\ul{x}^*$ on its boundary, corresponding to the unfinished strands in $\Gamma_\outside^{*,V,U}$, which are pairwise linked by an \emph{exterior planar link pattern} $\beta^*$ induced by the strands. We let $\Gamma_\inside^{*,V,U}$ be the collection of loops and strands of $\Gamma$ in $V^{*,U}$, which we call the \emph{unexplored part} of $\Gamma$.

\begin{theorem}[{\cite[Theorem~1.11]{amy-cle-resampling}}]
\label{thm:cle_partially_explored}
Consider the setup described above. The conditional law of $\Gamma_\inside^{*,V,U}$ given $\Gamma_\outside^{*,V,U}$ is a multichordal \clekp{} in $(V^{*,U};\ul{x}^*;\beta^*)$ with marked points $\ul{x}^*$ and exterior link pattern $\beta^*$ given as above.
\end{theorem}

Suppose that $\CI \subseteq \{1,\ldots,2N\}$, and let $\Gamma$ be a multichordal \clek{} in $(D;\ul{x};\beta)$ where $\ul{x}$ has $2N$ marked points. Let $\Upsilon_{\Gamma;\CI}$ denote the set of points that can be connected to the boundary arc from $x_{i}$ to $x_{i+1 \bmod 2N}$ for some $i \in \CI$ without crossing any loop or chord of $\Gamma$. In the case $N=0$, we let $\Upsilon_{\Gamma;\CI} = \Upsilon_\Gamma$ to be just the gasket of $\Gamma$.

\begin{proposition}[{\cite[Proposition~6.3]{amy-cle-resampling}}]\label{pr:mccle_gasket_tv_convergence}
 Let $N \in \N_0$, $\CI \subseteq \{1,\ldots,2N\}$. Suppose that $((D_n;\ul{x}_n))$ is a sequence that converges to $(D;\ul{x})$ in the Carathéodory topology where $\ul{x}$ (resp.\ $\ul{x}_N$) has $2N$ marked points. Let $\beta$ be an exterior planar link pattern between the marked points. Let $\Gamma$ (resp.\ $\Gamma_n$) have the law of a multichordal \clekp{} in $(D;\ul{x};\beta)$ (resp.\ $(D_n;\ul{x}_n;\beta)$). 
 Let $(U,V)\in\domainpair{D}$ with $V \Subset D$. Then the law of $((\Gamma_n)_\inside^{*,V,U}, \Upsilon_{\Gamma_n;\CI} \cap U)$ converges to the law of $(\Gamma_\inside^{*,V,U}, \Upsilon_{\Gamma;\CI} \cap U)$ in total variation as $n \to \infty$.
\end{proposition}

An important tool are \emph{local resamplings} of a \clekp{}. Suppose that we select small regions $(U,V)$ randomly and resample the \clekp{} in the unexplored part (which by Theorem~\ref{thm:cle_partially_explored} is a multichordal \clekp{}). Then there is a positive chance that we change the linking pattern of the loops that cross the region. This helps us to compare events for the \clekp{} with additional constraints on the global linking behavior of the loops. We recall a precise result below.

Let $\Gamma$ be a nested \clekp{} in $D$. A resampling of $\Gamma$ within a region $W \subseteq D$ is defined as follows. Suppose that we select $(U,V) \in \domainpair{D}$ with $V \subseteq W$ randomly in a way that is independent of the CLE configuration within $W$, and let $\wt{\Gamma}$ be the \clekp{} such that $\wt{\Gamma}_\outside^{*,V,U} = \Gamma_\outside^{*,V,U}$ and the remainder $\wt{\Gamma}_\inside^{*,V,U}$ is sampled from its conditional law given $\wt{\Gamma}_\outside^{*,V,U}$ (independently of $\Gamma_\inside^{*,V,U}$). Let $\Gamma^\resampled$ be the \clekp{} coupled with $\Gamma$ arising from repeating this procedure a random number of times.

\begin{lemma}[{\cite[Lemma~5.9]{amy-cle-resampling}}]
\label{lem:break_loops}
Suppose that $\Gamma$ is a nested \clekp{} in a simply connected domain $D$. For each $a\in(0,1)$, $b>0$ there exist $p > 0$ and $c>0$ such that the following holds. 

Let $z\in D$ and $j_0\in\Z$ be such that $B(z,2^{-j_0})\subseteq D$. There exists a resampling $\Gamma^\resampled_{z,j}$ of $\Gamma$ within $A(z,2^{-j},2^{-j+1})$ for each $j>j_0$ with the following property. Let $F_{z,j}$ be the event that
\begin{itemize}
 \item no loop in $\Gamma^\resampled_{z,j}$ crosses the annulus $A(z,2^{-j},2^{-j+1})$,
 \item denoting $\CC$ the loops in $\Gamma$ that cross $A(z,2^{-j},2^{-j+1})$, the collection of loops of $\Gamma$ and $\Gamma^\resampled_{z,j}$ remain the same in each connected component of $\C \setminus \bigcup\CC$ that is not surrounded by a loop in $\CC$ and intersects $B(z,2^{-j})$, and the gasket of $\Gamma$ in these components is contained in the gasket of $\Gamma^\resampled_{z,j}$.
\end{itemize}
For $k\in\N$, let $\wt G_{z,j_0,k}$ be the event that the number of $j=j_0+1,\ldots,j_0+k$ so that
\[ \p[ F_{z,j} \mid \Gamma ] \geq p\]
is at least $(1-a)k$.  Then
\[ \p[(\wt G_{z,j_0,k})^c] \le c e^{-b k} . \]
\end{lemma}

\subsection{Imaginary geometry}

We review some aspects of the coupling of $\SLE$ with the GFF developed in \cite{dub2009gff,s2016zipper,ms2016ig1, ms2016ig2, ms2016ig3, ms2017ig4} which will be used in this work. This coupling is useful because we can use it to localize certain events for SLE curves.

Suppose that $D \subseteq \C$ is a domain, and $h$ is a \emph{Gaussian free field} (GFF) on $D$. Recall that the GFF is the distribution valued Gaussian field with covariance function given by the Green's function $G$ for $\Delta$ on $D$, see \cite{s2007gff}.

Let $\kappa \in (0,4)$ and
\[ \chi = \frac{2}{\sqrt{\kappa}} - \frac{\sqrt{\kappa}}{2} . \]
A \emph{flow line} $\eta$ of $h$ refers to the solutions of the formal equation $\partial_t\eta(t) = e^{i h(\eta(t)) / \chi}$. It is shown in \cite{s2016zipper,dub2009gff,ms2016ig1} that formal solutions exist in a sense which we will review just below. For each $\kappa \in (0,4)$ we let
\[ \kappa' = \frac{16}{\kappa},\quad \lambda = \frac{\pi}{\sqrt{\kappa}},\quad\text{and}\quad \lambda' = \frac{\pi}{\sqrt{\kappa'}}.\]

\subsubsection{Flow lines}

Suppose that $h$ is a GFF on $\h$ with piecewise constant boundary values. The flow line $\eta$ of $\h$ from $0$ is characterized by the following property. Let $\tau$ be any stopping time for $\eta$. Then the conditional law of $h$ given $\eta\big|_{[0,\tau]}$ is that of a GFF in $\h \setminus \eta[0,\tau]$ with \emph{flow line boundary values}. The latter means that if $g_\tau\colon \h \setminus \eta[0,\tau] \to \h$ is the conformal transformation with $g_\tau(z) = z+o(1)$ as $z \to \infty$, then the boundary values of the field $h \circ g_\tau^{-1} - \chi \arg (g_\tau^{-1})'$ are $-\lambda$ (resp.\ $\lambda$) on the image of the left (resp.\ right) side of $\eta[0,\tau]$.

It is shown in \cite{ms2016ig1} that there exists a unique coupling of $h$ with $\eta$ as a flow line which is a certain \slekr{\ul{\rho}} process (depending on the boundary values of $h$) up until the continuation threshold is reached. Moreover, in the coupling $\eta$ is a.s.\ determined by $h$.

For each $x, \theta \in \R$ the flow line of $h$ with angle $\theta$ from $x$ is defined to be the flow line of $h(\cdot + x) + \theta \chi$. If $h$ is a GFF on a simply connected domain $D \subseteq \C$, then its flow line from $x$ to $y$ is defined as image under $\varphi$ of the flow line of the GFF $h \circ \varphi - \chi \arg \varphi'$ on $\h$ from $0$ to $\infty$ where $\varphi \colon \h \to D$ is a conformal transformation with $\varphi(0) = x$ and $\varphi(\infty) = y$.

We will use the following consequence of the reversibility of \slekr{\rho} \cite{ms2016ig2}. See \cite[Lemma~2.15]{amy2025tightness} for details. Recall the definition of $\angledouble$~\eqref{eq:angledouble}.

\begin{lemma}\label{le:reflected_fl_law}
Let $\kappa \in (2,4)$. 
Suppose $h$ is a GFF on $\h$ with boundary values $-\lambda(1+\rho_L)$  on $\R_-$ and $\lambda-\pi\chi = \lambda(\kappa/2-1)$ on $\R_+$ where $\rho_L > -2$. Let $\eta_1$ (resp.\ $\eta_2$) be the angle $0$ (resp.\ $-\angledouble$) flow line of $h$ from $0$ to $\infty$. Given $\eta_1$, let $\ol{\eta}_2$ be the angle $0$ flow line from $\infty$ to $0$ of the restriction of $h$ to the components of $\h\setminus\eta_1$ to the right of $\eta_1$ (i.e.\ $\ol{\eta}_2$ is reflected off $\eta_1$). Then the pair $(\eta_1,\eta_2)$ has the same law as $(\eta_1,\ol{\eta}_2)$ (modulo time reversal).
\end{lemma}

Similarly, it is shown in \cite{ms2017ig4} that flow lines from interior points can be defined. Let $z \in D$. The flow line $\eta$ of $h$ from $z$ is again characterized by the property that for any stopping time $\tau$, the conditional law of $h$ given $\eta\big|_{[0,\tau]}$ is that of a GFF in $D \setminus \eta[0,\tau]$ with flow line boundary values, only that in this case the values along $\eta$ are only specified up to a multiple of $2\pi\chi$. Again, $\eta$ is a.s.\ determined by $h$ in this coupling.

\subsubsection{Counterflow lines}

The GFF $h$ can also be coupled with an \slekp{} type process $\eta'$. In this case $\eta'$ is called a \emph{counterflow line} of $h$. It is characterized by the following property. Let $\tau$ be any stopping time for $\eta'$, and let $\h_\tau$ be the unbounded connected component of $\h \setminus \eta'[0,\tau]$. Then the conditional law given $\eta'|_{[0,\tau]}$ of the field $h \circ g_\tau^{-1} - \chi \arg (g_\tau^{-1})'$ is that of a GFF with boundary values given by $\lambda'$ (resp.\ $-\lambda'$) on the image of the left (resp.\ right) side of $\eta'[0,\tau]$.

It is shown in \cite{ms2016ig1} that there exists a unique coupling of $h$ with $\eta'$ as a counterflow line which is a certain \slekpr{\ul{\rho}} process (depending on the boundary values of $h$) up until the continuation threshold is reached. Moreover, in the coupling $\eta'$ is a.s.\ determined by $h$.

For each $x, \theta \in \R$ the counterflow line of $h$ with angle $\theta$ from $x$ is defined to be the counterflow line of $h(\cdot + x) + \theta \chi$. If $h$ is a GFF on a simply connected domain $D \subseteq \C$, then its counterflow line from $x$ to $y$ is defined as image under $\varphi$ of the counterflow line of the GFF $h \circ \varphi - \chi \arg \varphi'$ on $\h$ from $0$ to $\infty$ where $\varphi \colon \h \to D$ is a conformal transformation with $\varphi(0) = x$ and $\varphi(\infty) = y$.

\subsubsection{Interaction}

The manner in which the flow and counterflow lines of the GFF interact with each other is described in \cite{ms2016ig1,ms2017ig4}.  Let us first describe how the flow lines interact with each other.  Suppose that $\eta_1,\eta_2$ are flow lines with angles $\theta_1, \theta_2 \in \R$, respectively, which can start from boundary or from interior points. Suppose that $\eta_1$ intersects the right side of $\eta_2$. Then the angle difference between $\eta_1$ and $\eta_2$ is a value in $\theta_2-\theta_1+2\pi\chi\Z$. The following scenarios can occur. (i) The angle difference is in $(-\pi\chi,0)$, and $\eta_1$ crosses $\eta_2$ upon intersecting without subsequently crossing back. (ii) The angle difference is $0$, and $\eta_1$ merges with $\eta_2$. (iii) The angle difference is in $(0,2\lambda-\pi\chi)$, and $\eta_1$ bounces off $\eta_2$ without crossing.

The interaction between a flow line and a counterflow line is as follows. Suppose that $\eta'$ is a counterflow line of $h$ from $\infty$ to $0$.  Then the left (resp.\ right) boundary of $\eta'$ is equal to the flow line of $h$ from $0$ to $\infty$ with angle $\pi/2$ (resp.\ $-\pi/2$), where the points are visited in the reverse order.  More generally, any flow line of $h$ with angle in $[-\pi/2,\pi/2]$ is contained in the range of $\eta'$.

\subsubsection{Coupling with $\CLE_{\kappa'}$}

In the framework of imaginary geometry, the \clekp{} exploration tree is naturally coupled with a GFF as follows. Suppose that $h$ is a GFF on $\h$ with boundary conditions given by $\lambda'-\pi \chi$.  For each $x \in \R$ we let $\eta_x'$ be the counterflow line of $h$ from $\infty$ to $x$.  Then the collection $(\eta_x')$ is coupled together in exactly the same way as in the definition of the $\CLE_{\kappa'}$ exploration tree. Here each $\eta_x'$ is an $\SLE_{\kappa'}(\kappa'-6)$ in $\h$ from $\infty$ to $0$ where the force point is located infinitesimally to the right of $\infty$ (when standing at $\infty$ and looking towards $0$). This yields a natural coupling of $\CLE_{\kappa'}$ with $h$.

\subsection{Setup and notation for the proofs}
\label{se:setup_pf}

\newcommand*{\Esep}{E^{\mathrm{sep}}}
\newcommand*{\dsep}{r_{\mathrm{sep}}}

\newcommand*{\Ebreak}{E^{\mathrm{break}}}
\newcommand*{\Fbreak}{F^{\mathrm{break}}}
\newcommand*{\pbreak}{p_{\mathrm{break}}}

We discuss some of the setup and notation that we will repeatedly use in the proofs of this paper.

For $z \in \D$, $j\in\N$ we write
\[ A_{z,j} = A(z,2^{-j-1},2^{-j}) . \]
Recall the definition of the partially explored CLE $\Gamma_\outside^{*,B(z,3\cdot 2^{-j}),B(z,2\cdot 2^{-j})}$ and the region $B(z,3\cdot 2^{-j})^{*,B(z,2\cdot 2^{-j})}$ defined in Section~\ref{se:mcle}. Let $\alpha^*_{z,j}$ be the interior linking pattern in $B(z,3\cdot 2^{-j})^{*,B(z,2\cdot 2^{-j})}$ induced by $\Gamma_\inside^{*,B(z,3\cdot 2^{-j}),B(z,2\cdot 2^{-j})}$. We let
\begin{itemize}
 \item $\wt{\CF}_{z,j}$ be the $\sigma$-algebra generated by $\Gamma_\outside^{*,B(z,3\cdot 2^{-j}),B(z,2\cdot 2^{-j})}$.
 \item $\CF_{z,j}$ be the $\sigma$-algebra generated by $\Gamma_\outside^{*,B(z,3\cdot 2^{-j}),B(z,2\cdot 2^{-j})}$, $\alpha^*_{z,j}$, and $\metres{\C \setminus B(z,3\cdot 2^{-j})}{\cdot}{\cdot}{\Gamma}$.
\end{itemize}

The following lemma is a consequence of Theorem~\ref{thm:cle_partially_explored} and the Markovian property of the \clekp{} metric. It will be crucial for establishing independence across scales for the internal metrics.
\begin{lemma}
Let $U \subseteq B(z,2^{-j})$. The conditional law of $\metres{U}{\cdot}{\cdot}{\Gamma}$ given $\CF_{z,j}$ is given by sampling $U^*$ according to the law of a multichordal \clekp{} in $B(z,3\cdot 2^{-j})^{*,B(z,2\cdot 2^{-j})}$ conditionally on $\alpha^*_{z,j}$ and then sampling the internal metrics via the Markov property given $U^*$.
\end{lemma}

We will define several events that will be labeled $E^i_{z,j}$, $F^i_{z,j}$, $G^i_{z,j}$, etc. where $z\in\D$, $j\in\N$. We use $E^i_{z,j}$ to denote CLE events that are likely to occur, $F^i_{z,j}$ to denote events for resampled CLE that occur with positive probability given $E^i_{z,j}$, and $G^i_{z,j}$ to denote the events appearing in lemma statements. 
The implicit constants in the notation $O(\delta^b)$, $o^\infty(\delta)$, etc., never depend on $z,j$.

For $\dsep > 0$, let $\Esep_{z,j}$ denote the event that the marked points of $\Gamma_\outside^{*,B(z,3\cdot 2^{-j}),B(z,2\cdot 2^{-j})}$ are $\dsep$-separated in the sense that $\abs{y_i-y_{i'}} \ge \dsep 2^{-j}$ for each distinct pair $y_i,y_{i'}$ of its marked points. Recall that Proposition~\ref{pr:link_probability} implies that the conditional probability given $\wt{\CF}_{z,j}$ that $\alpha^*_{z,j}$ takes any given link pattern is bounded from below uniformly among $\dsep$-separated marked point configurations.

The following is proved in \cite{amy-cle-resampling}. (Strictly speaking, we have stated it there for sets of the form $K = \{ j_0,j_0+1,\ldots \}$, but it is easily seen that the proof also works in this general formulation.)

\begin{lemma}\label{le:separation_event}
For any $b>1$ there exists $\dsep>0$ and $c>0$ such that the following is true. Let $z\in\D$, and $K\subseteq\N$ a finite set such that $B(z,2^{-\min K}) \subseteq \D$. Then the probability that more than $1/10$ fraction of the events $(\Esep_{z,j})^c$ where $j \in K$ occur is at most $ce^{-b|K|}$.
\end{lemma}

For $\pbreak > 0$, we let $\Ebreak_{z,j}$ be the event described in Lemma~\ref{lem:break_loops} that for a resampling procedure within $A_{z,j-1}$ we have
\[ \p[ \Fbreak_{z,j} \mid \Gamma ] \one_{\Ebreak_{z,j}} \ge \pbreak \one_{\Ebreak_{z,j}} \]
where $\Fbreak_{z,j}$ is the event that all crossings of $A_{z,j-1}$ are broken up as described in Lemma~\ref{lem:break_loops}.

\newcommand*{\pmed}{p_{\mathrm{int}}}
\newcommand*{\intpts}[1]{\CX^{\mathrm{int}}_{#1}}

We recall the median $\median[\delta]{}$ and the quantiles $\quant[\delta]{q}{}$, $q \in (0,1)$, of particular random variables $D_\delta$, $\delta > 0$, defined in \cite[Section~3.1]{amy2025tightness}. The exact setup is not important, and we chose a concrete setup that is easy to work with.

Let $h$ be a GFF on $\h$ with boundary conditions given by $-\lambda - \angledouble \chi$ on~$\R_-$ and by $\lambda$ on $\R_+$ where $\angledouble$ is the double point angle~\eqref{eq:angledouble}. 
Let $\eta_1$ (resp.\ $\eta_2$) be the flow line of $h$ from $0$ to $\infty$ with angle $\angledouble$ (resp.\ $0$). Finally, we let $\Gamma$ be the collection of $\CLE_{\kappa'}$ that are generated by the restriction of $h$ to the components of $\h \setminus (\eta_1 \cup \eta_2)$ between $\eta_1$, $\eta_2$.

For $\delta \in (0,1)$ let $\varphi_\delta \colon \h \to \delta\D$ be a conformal map that takes $0$ to $-i \delta$ and $\infty$ to $i \delta$.  Let $h_\delta = h \circ \varphi_\delta^{-1} - \chi \arg (\varphi_\delta^{-1})'$, $\eta_i^\delta = \varphi_\delta(\eta_i)$ for $i=1,2$, let $\Gamma_\delta = \varphi_\delta(\Gamma)$.
Let
\begin{equation}\label{eq:pmed_def}
 \pmed = \p[ \eta_1^\delta \cap \eta_2^\delta \cap B(0,\delta/2) \neq \varnothing ] 
\end{equation}
which by scaling does not depend on $\delta$. Let $\intpts{\delta}$ be the set of pairs $(x,y) \in (\eta_1^\delta \cap \eta_2^\delta)^2$ such that the segments of $\eta_1^\delta,\eta_2^\delta$ between $x$ and $y$ are contained in $B(0,3\delta/4)$. Let $D_\delta$ have the law of the random variable
\begin{equation}\label{eq:quantile_def}
 \sup_{(x,y)\in\intpts{\delta}}\met{x}{y}{\Gamma_\delta} 
\quad\text{conditioned on the event that $\eta_1^\delta \cap \eta_2^\delta \cap B(0,\delta/2) \neq \varnothing$.}
\end{equation}
We let $\quant[\delta]{q}{}$, $q \in (0,1)$, be the $q$-quantile of $D_\delta$, and write $\median[\delta]{} = \quant[\delta]{1/2}{}$ for its median.

The following results have been proved in \cite{amy2025tightness}. The constants $\ddouble$, $d_\SLE$ are defined in~\eqref{eq:ddouble}, \eqref{eq:dsle}.

\begin{proposition}\label{pr:quantiles_metric}
We have
\[ \lambda^{d_\SLE+o(1)}\median[\delta]{} \le \median[\lambda\delta]{} \le \lambda^{\ddouble+o(1)}\median[\delta]{} \]
for any $\lambda,\delta \in (0,1)$. Further,
\[ \sup_{\delta\in (0,1)} \frac{\quant[\delta]{q}{}}{\quant[\delta]{q'}{}} < \infty \]
for any $q,q' \in (0,1)$.
\end{proposition}

Note that it will follow from the main result of this paper that $\median[\delta]{} = c\delta^\alpha$ for some $c>0$ where $\alpha$ is the exponent in Theorem~\ref{thm:exponent}.

We will sometimes use interior flow lines to ``detect'' regions that are bounded between \clekp{} loops. By this we mean the following (see, for instance, \cite[Section~4.2.2]{amy2025tightness} for more details). Suppose $U$ is a region (possibly consisting of infinitely many connected components) that is bounded between two flow lines $\eta_1,\eta_2$ where we allow the case when $\eta_2$ is reflected off $\eta_1$ in the opposite direction. Let $w_1,w_2 \in \C$ and let $\eta_{w_1},\eta_{w_2}$ be the flow lines with the same angles as $\eta_1$ (resp.\ $\eta_2$) starting from $w_1$ (resp.\ $w_2$) (we allow the case when $\eta_{w_2}$ is reflected off $\eta_{w_1}$ in the opposite direction). We then say that $\eta_{w_1},\eta_{w_2}$ \emph{detect $U$} if they merge into $\eta_1$ (resp.\ $\eta_2$) before they trace $\partial U$, and $\eta_{w_1},\eta_{w_2}$ do not intersect with any other angle differences than $\eta_1,\eta_2$ do.

\section{Ball crossing estimate}
\label{se:ball_crossing}

The purpose of this section is to prove Proposition~\ref{pr:ball_crossing} which we will repeatedly use in all proofs in this paper. Roughly speaking, we show that with high probability, there is a scale $2^{-j}$ where geodesic crossings of $B(z,2^{-j})$ are not much longer than expected.

\begin{proposition}\label{pr:ball_crossing}
Suppose we have the setup described at the beginning of Section~\ref{se:main_results} and $\met{\cdot}{\cdot}{\Gamma}$ is a weak geodesic \clekp{} metric. For every $b>1$ there exists $M>1$ and $c>0$ such that the following is true. Let $z\in\D$, and $K\subseteq\N$ a finite set such that $B(z,2^{-\min K}) \subseteq \D$. Let $G$ be the event that for at least a $9/10$ fraction of $j\in K$ it holds that for every admissible path $\gamma \subseteq B(z,2^{-j})$ and any $w_1,w_2 \in \gamma$ there exists an admissible path $\wt{\gamma} \subseteq B(z,2^{-j+1})$ from $w_1$ to $w_2$ with $\lmet{\wt{\gamma}} \le M\median[2^{-j}]{}$. Then
\[ \p[G^c] \le ce^{-b|K|} . \]
\end{proposition}

This proposition will be used as follows. Suppose that $b>1$, and $\dsep>0$ is chosen according to Lemma~\ref{le:separation_event}. Then with probability $1-O(e^{-bJ})$, for at least a $4/5$ fraction of scales $j = j_0,\ldots.,j_0+J$ both $\Esep_{z,j}$ and the event from Proposition~\ref{pr:ball_crossing} occur. Let $K \subseteq \N$, and suppose that we have another family of events $G_{z,j} \in \CF_{z,j+3}$ such that
\[ \p[G_{z,j}^c \mid \CF_{z,j}]\one_{\Esep_{z,j}} \le e^{-3b} \quad\text{for every}\quad z\in\D, j\in K.\]
Then
\[ \p\!\left[ \bigcap_{j\in K} (G_{z,j}^c \cap \Esep_{z,j}) \right] \le e^{-b|K|} . \]
Taking a union bound over all possible choices of $K$, we conclude that for every $z,\delta$ with $B(z,\delta) \subseteq \D$, with probability $1-O(\delta^{b-1})$ there exists (some random) $j_1 \in \{\log_2(\delta^{-1}),\ldots,\log_2(\delta^{-2})\}$ such that
\begin{itemize}
\item $G_{z,j_1}$ occurs, and
\item for every admissible path $\gamma \subseteq B(z,2^{-j_1})$ and any $w_1,w_2 \in \gamma$ there exists an admissible path $\wt{\gamma} \subseteq B(z,2^{-j_1+1})$ from $w_1$ to $w_2$ with $\lmet{\wt{\gamma}} \le M\median[2^{-j_1}]{}$.
\end{itemize}

We observe that it suffices to prove a variant of Proposition~\ref{pr:ball_crossing} where we only require the good event to hold on at least one scale $j$ (instead of $9/10$ fraction of scales).

\begin{lemma}\label{le:ball_crossing_one}
Suppose we have the setup described at the beginning of Section~\ref{se:main_results} and $\met{\cdot}{\cdot}{\Gamma}$ is a weak geodesic \clekp{} metric. For every $b>1$ there exists $M>1$ and $c>0$ such that the following is true. Let $z\in\D$, and $K\subseteq\N$ a finite set such that $B(z,2^{-\min K}) \subseteq \D$. Let $G^1$ be the event that there exists $j\in K$ such that for every admissible path $\gamma \subseteq B(z,2^{-j})$ and any $w_1,w_2 \in \gamma$ there exists an admissible path $\wt{\gamma} \subseteq B(z,2^{-j+1})$ from $w_1$ to $w_2$ with $\lmet{\wt{\gamma}} \le M\median[2^{-j}]{}$. Then
\[ \p[(G^1)^c] \le ce^{-b|K|} . \]
\end{lemma}

\begin{proof}[Proof of Proposition~\ref{pr:ball_crossing} assuming Lemma~\ref{le:ball_crossing_one}]
If $G$ does not occur, there exists $K' \subseteq K$ with $|K'| \ge |K|/10$ such that the event fails in each $j\in K'$. There are at most $2^{|K|}$ possible choices for such $K'$. Since we can suppose $b$ is large, the result follows by applying Lemma~\ref{le:ball_crossing_one} to each possible choice of $K'$ and taking a union bound.
\end{proof}

In the remainder of this section, we prove Lemma~\ref{le:ball_crossing_one}. The proof involves several steps. If the event $G$ only involved a single annulus $A_{z,j_1}$, then the statement would follow from an independence across scales argument. However, we need to bound the lengths of crossings of the ball $B(z,2^{-j_1})$ and not just the annulus $A_{z,j_1}$, and these paths may cross also subsequent annuli. We need to argue that the crossings of the subsequent annuli are very unlikely to contribute much more to the length. For this, we use the superpolynomial upper tail for the metric proved in \cite{amy2025tightness}.

Recall the $\sigma$-algebra $\CF_{z,j}$ and the events $\Esep_{z,j}$, $\Ebreak_{z,j}$ (depending on $\dsep>0$, $\pbreak>0$) defined in Section~\ref{se:setup_pf}.

\begin{lemma}\label{le:annulus_crossing_conditional}
For each $\dsep>0$, $\pbreak>0$, and $b>1$ there exists a constant $c>0$ such that the following is true. Let $z \in \D$, $j < j_1 < j_2 \le \infty$ such that $B(z,2^{-j+2}) \subseteq \D$. Let $G$ denote the event that for any $w_1,w_2 \in A(z,2^{-j_2},2^{-j_1}) \cap \Upsilon_{\Gamma}$ that are connected in $A(z,2^{-j_2},2^{-j_1}) \cap \Upsilon_{\Gamma}$ there is an admissible path $\gamma \subseteq A(z,2^{-j_2-1},2^{-j_1+1})$ from $w_1$ to $w_2$ with $\lmet{\gamma} \le M\median[2^{-j_1}]{}$. Then
\[
 \p[ G^c \cap \Ebreak_{z,j_1-2} \mid \CF_{z,j} ] \one_{\Esep_{z,j}} \le cM^{-b} .
\]
\end{lemma}

Using a resampling argument, we reduce this to the case of a \clekp{}. Let $D \subseteq \D$ be a simply connected domain, and let $\Gamma_D$ be a \clekp{} in $D$. We consider the gasket of $\Gamma_D$ (here we mean the cluster adjacent to the boundary $\partial D$, in contrast to the setup of the main theorems). For every $U \Subset D$, the internal metric $\metres{U}{\cdot}{\cdot}{\Gamma_D}$ is defined due to absolute continuity. (In fact, our domains of interest will be complementary connected components of \clekp{} loops, so in that case the metric on all $\ol{D} \cap \Upsilon_{\Gamma_D}$ is defined.)

\begin{lemma}\label{le:annulus_crossing_cle}
For each $b>1$ there exists a constant $c>0$ such that the following is true. Let $D \subseteq \D$ be open, simply connected, and $z \in \D$, $j_1 < j_2 \le \infty$ such that $B(z,2^{-j_1+2}) \subseteq D$. Let $\Gamma_D$ be a \clekp{} in $D$. Given $\Gamma_D$, sample the internal metric $\metres{B(z,2^{-j_1+1})}{\cdot}{\cdot}{\Gamma_D}$. Let $G$ denote the event that for any $w_1,w_2 \in A(z,2^{-j_2},2^{-j_1}) \cap \Upsilon_{\Gamma_D}$ that are connected in $A(z,2^{-j_2},2^{-j_1}) \cap \Upsilon_{\Gamma_D}$ there is an admissible path $\gamma \subseteq A(z,2^{-j_2-1},2^{-j_1+1})$ from $w_1$ to $w_2$ with $\lmet{\gamma} \le M\median[2^{-j_1}]{}$. Then
\[ \p[G^c] \le cM^{-b} . \]
\end{lemma}

\begin{proof}
 This follows from the proof of \cite[Proposition~5.17]{amy2025tightness}. Note that although \cite[Proposition~5.17]{amy2025tightness} is stated for the setup described described at the beginning of Section~\ref{se:main_results}, the proof applies simultaneously to each cluster and in domains $D \subseteq \D$ when we stay away from the boundary of $D$.
\end{proof}

\begin{lemma}\label{le:annulus_crossing_conditional_wo_lp}
 In the setup of Lemma~\ref{le:annulus_crossing_conditional} we have
 \[
 \p[ G^c \cap \Ebreak_{z,j_1-2} \mid \wt{\CF}_{z,j} ] \le cM^{-b} .
 \]
\end{lemma}

\begin{proof}
 Let $\Gamma^\resampled$ denote the \clekp{} arising from the resampling procedure in $A_{z,j_1-2}$ in Lemma~\ref{lem:break_loops}. Let $\Fbreak_{z,j_1-2}$ be the event that all crossings of $A_{z,j_1-2}$ are broken up in $\Gamma^\resampled$, and that $B(z,2^{-j_1+1}) \cap \Upsilon_\Gamma \subseteq B(z,2^{-j_1+1}) \cap \Upsilon_{\Gamma^\resampled}$. By the definition of $\Ebreak_{z,j_1-2}$, we have
 \[ \p[\Fbreak_{z,j_1-2} \mid \Gamma] \one_{\Ebreak_{z,j_1-2}} \ge \pbreak \one_{\Ebreak_{z,j_1-2}} . \]
By the Markovian property of the \clekp{} metric, we can couple the internal metrics so that $\metres{B(z,2^{-j_1+1})}{\cdot}{\cdot}{\Gamma} = \metres{B(z,2^{-j_1+1})}{\cdot}{\cdot}{\Gamma^\resampled}$ on $B(z,2^{-j_1+1}) \cap \Upsilon_\Gamma$. Hence, if we let $\wt{G}$ be defined in the same way as $G$ but for $\Gamma^\resampled$, then
 \[ \p[\wt{G} \cap \Fbreak_{z,j_1-2} \mid \wt{\CF}_{z,j}] \ge \pbreak \p[G \cap \Ebreak_{z,j_1-2} \mid \wt{\CF}_{z,j}] . \]
 On the other hand, on the event $\Fbreak_{z,j_1-2}$, if we condition on the CLE loops that intersect the complement of $B(z,2^{-j_1+2})$, then the remainder is a \clekp{} in a random region $D$ containing $B(z,2^{-j_1+1})$, and the internal metric within $B(z,2^{-j_1+1})$ is conditionally independent of the internal metric outside $D$. Therefore by Lemma~\ref{le:annulus_crossing_cle} we have
 \[ \p[\wt{G} \cap \Fbreak_{z,j_1-2} \mid \wt{\CF}_{z,j}] \lesssim M^{-b} \]
 concluding the proof.
\end{proof}

\begin{proof}[Proof of Lemma~\ref{le:annulus_crossing_conditional}]
 By the Markovian property of the metric, conditioning on the internal metric outside $B(z,3\cdot 2^{-j})$ does not affect the conditional law, so the only extra conditioning compared to Lemma~\ref{le:annulus_crossing_conditional_wo_lp} is on the link pattern $\alpha^*_{z,j}$. But on the event $\Esep_{z,j}$, by Proposition~\ref{pr:link_probability}, the conditional probability of each link pattern is bounded from below by a positive constant that depends only on $\dsep$. The result follows.
\end{proof}

\begin{lemma}\label{le:loop_crossings_cond}
Suppose that $B(z,2^{-j+2}) \subseteq \D$. Let $j' \ge j+2$, and $G$ denote the event that the total number of crossings of $A_{z,j'}$ by all loops of $\Gamma$ is at most $M$. Then $\p[G^c] = o^\infty(M^{-1})$ and $\p[ G^c \cap \Ebreak_{z,j'-2} \mid \CF_{z,j} ] \one_{\Esep_{z,j}} = o^\infty(M^{-1})$ where the implicit constants do not depend on $z,j,j'$. 
\end{lemma}

\begin{proof}
 Without the conditioning, this is a consequence of Lemma~\ref{le:loop_crossings_tail}. The generalization to the conditional probability goes by the exact same argument as in the proof of Lemma~\ref{le:annulus_crossing_conditional}.
\end{proof}

\begin{proof}[Proof of Lemma~\ref{le:ball_crossing_one}]
Let $b > 1$ be given. By Lemmas~\ref{lem:break_loops} and~\ref{le:separation_event}, we can choose $\dsep>0$, $\pbreak > 0$, and $c_0>0$ such that the probability that $\Esep_{z,j} \cap \Ebreak_{z,j}$ occurs for less than $1/2$ fraction of the annuli $A_{z,j}$, $j\in K$, is $O(e^{-b\abs{K}})$. There are at most $2^{|K|}$ choices for the scales $K' \subseteq K$ on which the events occur. Since we can take a union bound over all choices, we will fix a particular choice $K' \subseteq K$ and restrict to the event that $E_{K'} = \bigcap_{j\in K'} \Esep_{z,j} \cap \Ebreak_{z,j}$ occurs.

Fix $\zeta \in (0,\ddouble/2)$. For $M>0$ and each $j' \ge j+3$, let $G^{j}_{j'}$ be the event that
\begin{itemize}
\item for each admissible path $\gamma \subseteq A(z,2^{-j'-1},2^{-j'+1})$ and $w_1,w_2 \in \gamma$ there exists an admissible path $\wt{\gamma}$ from $w_1$ to $w_2$ with $\wt{\gamma} \subseteq A(z,2^{-j'-2},2^{-j'+2})$ and
\[ \lmett{\wt{\gamma}} \le M2^{\zeta(j'-j)}\median[2^{-j'}]{} \le M2^{-\zeta(j'-j)}\median[2^{-j}]{} \]
(where the last inequality follows from Proposition~\ref{pr:quantiles_metric}),
\item the total number of crossings of $A_{z,j}$ by all loops is at most $M2^{a(j'-j)}$.
\end{itemize}
Further, let
\[ \ol{G}^{j} := \bigcap_{j' \ge j+3} G^{j}_{j'} . \]
By Lemmas~\ref{le:annulus_crossing_conditional} and~\ref{le:loop_crossings_cond}, for each $j \in K'$ we have
\begin{equation}\label{eq:subsequent_annulus}
\p[(G^{j}_{j'})^c \cap E_{K'} \mid \CF_{z,j}] = O(M^{-b}e^{-b(j'-j)}) .
\end{equation}
Hence, by taking a union bound over $j' \geq j$ hence summing the right hand side of~\eqref{eq:subsequent_annulus} over $j' \geq j$, we can choose $M$ so that
\begin{equation*}
\p[(\ol{G}^{j})^c \cap E_{K'} \mid \CF_{z,j}] \le e^{-b} .
\end{equation*}
We claim that
\begin{equation}\label{eq:successful_streak}
\p\!\left[ \left( \bigcup_{j\in K'} \ol{G}^{j} \right)^c \cap E_{K'} \right] = O(e^{-b|K'|}) .
\end{equation}
We argue that on the event $\ol{G}^{j}$ the statement of Lemma~\ref{le:ball_crossing_one} holds with $j+4$ in place of $j$ and $K+4$ in place of $K$. Indeed, let $\gamma \subseteq B(z,2^{-j-4})$ be an admissible path. We decompose $\gamma$ into its initial segment, its final segment, and inbetween segments $\gamma_l$ where each $\gamma_l$ goes from $\partial B(z,2^{-j'})$ for some $j' > j+4$ to its next hitting point of $\partial A(z,2^{-j'-1},2^{-j'+1})$. For each such $j'$ and each corresponding segment $\gamma_l$, there is an admissible path $\wt{\gamma}_l$ with the same starting point and endpoint such that $\lmett{\wt{\gamma}_l} \le M2^{-\zeta(j'-j)}\median[2^{-j}]{}$. Similarly, if two such segments $\gamma_l$, $\gamma_{l'}$, $l\le l'$, are connected within $A(z,2^{-j'-1},2^{-j'+1}) \cap \Upsilon_\Gamma$, then there is an admissible path $\wt{\gamma}_l$ going from the first point of $\gamma_l$ to the last point of $\gamma_{l'}$ with $\lmett{\wt{\gamma}_l} \le M2^{-\zeta(j-j_0)}\median[2^{-j_0}]{}$. This bypasses all the segments of $\gamma$ between $\gamma_l$ and $\gamma_{l'}$. Note that if both $\gamma_l$, $\gamma_{l'}$ go from $\partial B(z,2^{-j'})$ to $\partial B(z,2^{-j'-1})$ (resp.\ to $\partial B(z,2^{-j'+1})$), then they are connected within $A_{z,j'}$ (resp.\ $A_{z,j'-1}$) unless there is a loop crossing the annulus that separates them. Since on the event $\ol{G}^{j}$ for each $j' \ge j+3$ the total number of crossings of $A_{z,j'}$ by loops is at most $M2^{a(j'-j)}$, the number of such segments $\gamma_l$ that are not connected is at most $2M2^{a(j-j_0)}$. Hence, we can find for each $j' \ge j+4$ a collection of at most $O(M2^{a(j'-j)})$ such paths $\wt{\gamma}_l$, and concatenating them gives us an admissible path $\wt{\gamma}$ going from the first to the last point of $\gamma$ with total length at most
\[ 
\lmett{\wt{\gamma}} 
\lesssim \sum_{j'\ge j+3}M2^{a(j'-j)}M2^{-\zeta(j'-j)}\median[2^{-j}]{}
\lesssim M^2\median[2^{-j}]{} .
\]
To complete the proof, we need to show~\eqref{eq:successful_streak}. For fixed $j$, if $\ol{G}^{j}$ does not occur, then there is a first $j'\ge j+3$ where $G^{j}_{j'}$ fails. The event $G^{j}_{j'}$ is measurable with respect to $\CF_{z,j'+3}$. We will retry inside $B(z,2^{-j'-3})$. Concretely, if we let $j_1 \ge j'+3$ be the next integer that is in $K'$, then
\[ \p[ (\ol{G}^{j_1})^c \cap E_{K'} \mid \CF_{z,j'+3} ] \le e^{-b} . \]
If we do not succeed for any $j \in K'$, there must be $n_1,n_2,\ldots$ with $n_1+n_2+\cdots \ge |K'|/10$ and $j_1,j_2,\ldots \in K'$ with $j_l \ge j_{l-1}+n_l+3$ for each $l$ such that $(G^{j_l}_{j_l+n_{l+1}})^c$ occurs (since we have skipped at most $2$ scales in $K'$ upon each failure). For any such collection of failure points, by~\eqref{eq:subsequent_annulus},
\[ \p\!\left[ \bigcap_l (G^{j_l}_{j_l+n_{l+1}})^c \cap E_{K'} \right] \le e^{-b(n_1+n_2+\ldots)} \le e^{-b|K'|} . \]
Since there are at most $2^{|K'|}$ choices for such sequences $(n_l)$, we obtain~\eqref{eq:successful_streak} via a union bound (since we can assume $b$ is large).
\end{proof}

\section{Bi-Lipschitz equivalence}
\label{sec:bi_lip}

In this section we suppose that we have the setup described at the beginning of Section~\ref{se:main_results}, and $\met{\cdot}{\cdot}{\Gamma}$ and $\mett{\cdot}{\cdot}{\Gamma}$ are two weak geodesic \clekp{} metrics coupled so that they are conditionally independent given $\Gamma$. The conditional independence is useful so that (in the setup of Section~\ref{se:setup_pf}) the conditional law of $\met{\cdot}{\cdot}{\Gamma}$ given $\CF_{z,j}$ is not affected by further conditioning on $\mettres{\C \setminus B(z,3\cdot 2^{-j})}{\cdot}{\cdot}{\Gamma}$, and vice versa. Note that since $\met{\cdot}{\cdot}{\Gamma}$ takes values in a Polish space, its regular conditional law given $\Gamma$ exists, hence we can couple $\met{\cdot}{\cdot}{\Gamma}$, $\mett{\cdot}{\cdot}{\Gamma}$ this way.

We assume further that there is a constant $c>1$ such that
\begin{equation}\label{eq:same_medians_ass}
c^{-1}\median[\delta]{} \le \mediant[\delta]{} \le c\,\median[\delta]{}
\quad\text{for all } \delta \in (0,1]
\end{equation}
where $\median[\delta]{}, \mediant[\delta]{}$ are the medians for the two metrics recalled above Proposition~\ref{pr:quantiles_metric}. Our goal is to show they are a.s.\ bi-Lipschitz equivalent with deterministic constants.

\begin{proposition}\label{pr:bilipschitz}
In the setup above, assume~\eqref{eq:same_medians_ass}. There exists a deterministic constant $M$ (depending on the laws of $\met{\cdot}{\cdot}{\Gamma}$, $\mett{\cdot}{\cdot}{\Gamma}$) such that almost surely
\[ \mett{x}{y}{\Gamma} \le M\met{x}{y}{\Gamma} \quad\text{for all}\quad x,y \in \Upsilon_\Gamma .\]
\end{proposition}

The idea to prove Proposition~\ref{pr:bilipschitz} is that with overwhelming probability, for a positive fraction of sufficiently small scales the lengths of geodesic crossings of small balls under both metrics are comparable to their median. We have already proved an upper bound for the geodesic lengths in Proposition~\ref{pr:ball_crossing}. We now give a corresponding lower bound. Recall the $\sigma$-algebra $\CF_{z,j}$ and the events $\Esep_{z,j}$, $\Ebreak_{z,j}$ (depending on $\dsep>0$, $\pbreak>0$) defined in Section~\ref{se:setup_pf}.

\begin{lemma}\label{le:len_lb_conditional}
Suppose that $\met{\cdot}{\cdot}{\Gamma}$ is a weak geodesic \clekp{} metric. Let $d_{z,j}$ denote the infimum of $\lmet{\gamma}$ over all admissible crossings $\gamma$ of $A_{z,j}$. 
For every $\dsep>0$, $\pbreak>0$, and $q > 0$ there exists $m > 0$ such that
\[ \p[ d_{z,j} < m\,\median[2^{-j}]{} ,\, \Ebreak_{z,j-1} \mid \CF_{z,j} ] \one_{\Esep_{z,j}} \le q \]
for each $z,j$ with $B(z,2^{-j+2}) \subseteq \D$.
\end{lemma}

By the same argument as in the proof of Lemma~\ref{le:annulus_crossing_conditional}, it suffices to prove the unconditional version of Lemma~\ref{le:len_lb_conditional}.

\begin{lemma}\label{le:len_lb}
 For every $q > 0$ there exists $m > 0$ such that the following is true. Let $D \subseteq \D$ be open, simply connected, and $z \in \D$, $j \in \N$ such that $\distE(z,\partial D) \in [2^{-j+1},2^{-j+2}]$. Let $\Gamma_D$ be a \clekp{} in $D$. Given $\Gamma_D$, sample the internal metric $\metres{B(z,2^{-j})}{\cdot}{\cdot}{\Gamma_D}$. Let $d_{z,j}$ denote the infimum of $\lmet{\gamma}$ over all admissible crossings $\gamma$ of $A_{z,j}$. Then
 \[ \p[ d_{z,j} < m\,\median[2^{-j}]{} ] \le q . \]
\end{lemma}

We prove Lemma~\ref{le:len_lb} by transferring it to the setup of \cite[Section~3.1]{amy2025tightness}, recalled in Section~\ref{se:setup_pf}, where we can use the comparability of quantiles to conclude. To keep the argument clean, we carry out the argument step by step although each step is very similar to the others.

\begin{lemma}\label{le:lb_fl}
 Let $h$ be a GFF in $\D$ with boundary values so that the following flow lines are defined. Let $\eta_1$ (resp.\ $\eta_2$) be the flow line of $h$ with angle $0$ (resp.\ $-\angledouble$) from $-i$ to $i$. For each $c>0$ and $q>0$ there is $m>0$ such that the following is true for each $\delta \in (0,1]$. Let $h^\delta = h(\delta^{-1}\cdot)$ and $\eta_1^\delta = \delta\eta_1$, $\eta_2^\delta = \delta\eta_2$. Sample a \clekp{} $\Gamma^\delta$ and the internal metric in the regions bounded between $\eta_1^\delta,\eta_2^\delta$. Let $G$ be the event that for every $x,y \in \eta_1^\delta \cap \eta_2^\delta$, if the segments of $\eta_1^\delta,\eta_2^\delta$ from~$x$ to~$y$ are contained in $B(0,\delta/2)$ and have Euclidean diameter at least $c\delta$, then
 \[
  \met{x}{y}{\Gamma^\delta} \ge m\,\median[\delta]{} .
 \]
 Then
 \[
  \p[G^c] \le q .
 \]
\end{lemma}

\begin{proof}
 The setup is very similar to \cite[Section~3.1]{amy2025tightness}, and we will use a resampling procedure to reduce it to that setup so that the claim will follow from the comparability of the quantiles.
 
 Let $q>0$ be given. First, let us notice that there are only finitely many crossings of $B(0,\delta/2)$ by $\eta_1^\delta,\eta_2^\delta$ of Euclidean diameter at least $c\delta$, therefore we can choose $N$ large enough so that the probability that there are more than $N$ such crossings is at most $q$ (which is independent of $\delta$ by scale invariance). Next, let us sample $z_1,z_2 \in B(0,\delta/2)$ randomly according to Lebesgue measure and let $\eta^1_{z_1}$ (resp.\ $\eta^2_{z_2}$) be the flow line with angle $0$ (resp.\ $-\angledouble$) starting from $z_1$ (resp.\ $z_2$) and stopped at a random stopping time before exiting $B(0,\delta/2)$ (e.g.\ sampling $Y \in \N$ randomly and stopping after it has created $Y$ subsegments of diameter $(c/2)\delta$).
 
 For each $x,y \in \eta_1^\delta \cap \eta_2^\delta$, let $U_{x,y}$ be the regions bounded between the segments of $\eta_1^\delta,\eta_2^\delta$ from $x$ to $y$. For some $p_1,p_2>0$ let $E_\delta$ be the event that for each $x,y \in \eta_1^\delta \cap \eta_2^\delta$ with $U_{x,y} \subseteq B(0,\delta/2)$ and $\diamE(U_{x,y}) \ge (c/2)\delta$, the conditional probability of the following is at least $p_1$:
 \begin{itemize}
  \item $\eta^1_{z_1},\eta^2_{z_2}$ detect $U_{x,y}$ (in the sense described at the end of Section~\ref{se:setup_pf}).
  \item If we let $\wt{h}^\delta$ be a GFF in $\delta\D \setminus (\eta^1_{z_1} \cup \eta^2_{z_2})$ with boundary values on $\partial(\delta\D)$ as described before~\eqref{eq:pmed_def} and the same values as $h^\delta$ along $\eta^1_{z_1},\eta^2_{z_2}$, and let $\wt{\eta}_1^\delta,\wt{\eta}_2^\delta$ be its corresponding flow lines from $-i\delta$ to $i\delta$, then the probability that they merge into $\eta^1_{z_1},\eta^2_{z_2}$ and do not intersect in $B(0,\delta/2)$ except on $\eta^1_{z_1},\eta^2_{z_2}$ is at least $p_2$.
 \end{itemize}
 Since all these conditions are scale-invariant, we can choose $p_1>0$ small so that $\p[E_\delta^c] \le q$ for each~$\delta$.
 
 Let $G^\bad$ be the event that if we sample $\eta^1_{z_1},\eta^2_{z_2}$ as described above and sample the internal metric in the region $\wt{U}$ bounded between $\eta^1_{z_1},\eta^2_{z_2}$, then the event in the second bullet point occurs and $\metres{\wt{U}}{\wt{x}}{\wt{y}}{\wt{\Gamma}} < m\,\median[\delta]{}$ where $\wt{x},\wt{y}$ are the first (resp.\ last) intersection point. Due to the Markovian property of the metric, we can define the internal metric in $\wt{U}$ and couple it so that it agrees with $\metres{U_{x,y}}{\cdot}{\cdot}{\Gamma}$ on the event $U_{x,y} \subseteq \wt{U}$. Indeed, the conditional law of $\metres{U_{x,y}}{\cdot}{\cdot}{\Gamma}$ depends only on $\eta_1^\delta,\eta_2^\delta$ and not on whether $\eta^1_{z_1},\eta^2_{z_2}$ detect $U_{x,y}$. Moreover, the conditional law of $\metres{\wt{U}}{\cdot}{\cdot}{\Gamma}$ depends only on the values of the GFF within $B(0,\delta/2)$. We therefore have, by the first item in the definition of $E_\delta$,
 \[ \p[G^\bad] \ge p_1\p[G^c \cap E_\delta] . \] 
 Let $\wt{\p}$ denote the law of a GFF on $\delta\D$ with boundary values as described before~\eqref{eq:pmed_def}. By local absolute continuity, $\p[G^\bad] \lesssim \wt{\p}[G^\bad]^{1/2}$ where the implicit constant does not depend on $\delta$ (see \cite[Lemma~2.18]{amy2025tightness} where it is argued that the local absolute continuity extends to the case when the metric is random). Finally, if we let $\intpts{\delta}$ be as defined above~\eqref{eq:quantile_def}, we have
 \begin{equation}\label{eq:lb_pf}
  \wt{\p}\!\left[ \sup_{(x_0,y_0)\in\intpts{\delta}}\met{x_0}{y_0}{\Gamma_\delta} < m\,\median[\delta]{} \right] \ge p_2\wt{\p}[G^\bad] .
 \end{equation}
 Let $q' = p_1^2 p_2 q^2$. By Proposition~\ref{pr:quantiles_metric}, we can find $m > 0$ so that $\inf_{\delta\in(0,1]} \frac{\quant[\delta]{q'}{}}{\median[\delta]{}} \ge m$. In particular, the left-hand side of~\eqref{eq:lb_pf} is at most $q'$, which combined with the previous comparisons leads to
 \[ \p[G^c \cap E_\delta] \lesssim p_1^{-1}p_2^{-1/2}(q')^{1/2} \lesssim q , \]
 and finally
 \[ \p[G^c] \lesssim 2q . \]
\end{proof}

\begin{lemma}\label{le:lb_fl_refl}
 Let $h$ be a GFF in $\D$ with boundary values so that the following flow lines are defined. Let $\eta_1$ be the flow line of $h$ from $-i$ to $i$. Given $\eta_1$, let $\eta_2$ be the flow line of $h$ in the components of $\D \setminus \eta_1$ to the right of $\eta_1$ from $i$ to $-i$, reflected off $\eta_1$ in the opposite direction. For each $c>0$ and $q>0$ there is $m>0$ such that the following is true for each $\delta \in (0,1]$. Let $h^\delta = h(\delta^{-1}\cdot)$ and $\eta_1^\delta = \delta\eta_1$, $\eta_2^\delta = \delta\eta_2$. Sample a \clekp{} $\Gamma^\delta$ and the internal metric in the regions bounded between $\eta_1^\delta,\eta_2^\delta$. Let $G$ be the event that for every $x,y \in \eta_1^\delta \cap \eta_2^\delta$, if the segments of $\eta_1^\delta,\eta_2^\delta$ from $x$ to $y$ are contained in $B(0,\delta/2)$ and have Euclidean diameter at least $c\delta$, then
 \[
  \met{x}{y}{\Gamma^\delta} \ge m\,\median[\delta]{} .
 \]
 Then
 \[
  \p[G^c] \le q .
 \]
\end{lemma}

\begin{proof}
 In case the boundary values of $h$ are as in Lemma~\ref{le:reflected_fl_law}, then this follows from Lemma~\ref{le:lb_fl} by the reversal symmetry described in Lemma~\ref{le:reflected_fl_law}. To transfer the result to general boundary values, we can use an analogous argument as in the proof of Lemma~\ref{le:lb_fl}.
\end{proof}

\begin{proof}[Proof of Lemma~\ref{le:len_lb}]
 The argument is similar to the proof of Lemma~\ref{le:lb_fl}, so we will be brief. Let $q>0$ be given.
 
 By \cite[Lemma~5.14]{amy-cle-resampling}, there almost surely exists a finite set of points $w_i$ that separate $\bIn A_{z,j}$ from $\bOut A_{z,j}$ in $A_{z,j} \cap \Upsilon_{\Gamma_D}$. Sample two points $u,v \in A_{z,j}$ randomly according to Lebesgue measure. Let $\eta_{u}$ be the flow line starting from $u$, and let $\eta_{u,v}$ be the flow line starting from $v$, reflected off $\eta_{u}$ in the opposite direction. For $n \in \N$, $M,p_1,p_2 > 0$, let $E$ be the event that there are a collection of at most $n$ separating points $w_i$ as described above and for each $w_i$ there is a region $U_i$ bounded between the two strands intersecting at $w_i$ such that the following hold:
 \begin{itemize}
  \item The two endpoints of $U_i$ have Euclidean distance at least $M^{-1}2^{-j}$.
  \item The conditional probability given $\Gamma_D$ is at least $p_1$ that
  \begin{itemize}
   \item $\eta_{u},\eta_{u,v}$ detect $U_i$.
   \item If we sample a GFF in $B(z,2^{-j+2}) \setminus (\eta_{u} \cup \eta_{u,v})$ with boundary values on $\partial B(z,2^{-j+2})$ as in Lemma~\ref{le:lb_fl_refl} and same values as $h$ along $\eta_{u},\eta_{u,v}$, then the probability that the flow lines of $\wt{h}$ from $z-i2^{-j+2}$ (resp.\ $z+i2^{-j+2}$) merge into $\eta_{u},\eta_{u,v}$ is at least $p_2$.
  \end{itemize}
 \end{itemize}
 Note that the event $E$ depends only on $\Gamma_D$ and not on the metric $\met{\cdot}{\cdot}{\Gamma_D}$. By Proposition~\ref{pr:mccle_gasket_tv_convergence} (and using that $\distE(z,\partial D) \in [2^{-j+1},2^{-j+2}]$), we can choose $n \in \N$, $M>0$ large enough and $p_1,p_2 > 0$ small enough so that $\p[E^c] \le q$ uniformly in the choice of $z,j$, and $D$.
 
 Following the exact same argument as in the proof of Lemma~\ref{le:lb_fl} we can find, now using Lemma~\ref{le:lb_fl_refl} as input, some $m>0$ such that
 \[ \p[ d_{z,j} < m\,\median[2^{-j}]{} ,\,E ] \le q , \]
 and hence
 \[ \p[ d_{z,j} < m\,\median[2^{-j}]{} ] \le 2q . \]
\end{proof}

\begin{proof}[Proof of Proposition~\ref{pr:bilipschitz}]
We are going to use Proposition~\ref{pr:ball_crossing} and Lemma~\ref{le:len_lb_conditional} to find a deterministic constant $M$ such that almost surely the plane is covered by small balls in which the geodesics with respect to $\met{\cdot}{\cdot}{\Gamma}$ and $\mett{\cdot}{\cdot}{\Gamma}$ have comparable length.

Applying Proposition~\ref{pr:ball_crossing} with $b=8$ (say), there is $M>1$ such that with probability $1-O(\delta^8)$ the event from Proposition~\ref{pr:ball_crossing} occurs for $\mett{\cdot}{\cdot}{\Gamma}$ on at least $9/10$ fraction of scales $j \in \{\log_2(\delta^{-1}),\ldots,\log_2(\delta^{-2})\}$. By the same argument (c.f.\ the paragraph below Proposition~\ref{pr:ball_crossing}), letting $d_{z,j}$ as in Lemma~\ref{le:len_lb_conditional}, there is $m>0$ such that with probability $1-O(\delta^8)$ we have $d_{z,j} \ge m\,\median[2^{-j}]{}$ on at least $9/10$ fraction of scales. In particular, on at least $4/5$ fraction of scales $j \in \{\log_2(\delta^{-1}),\ldots,\log_2(\delta^{-2})\}$ we have
\begin{itemize}
\item $d_{z,j} \ge m\,\median[2^{-j}]{}$, and
\item for every admissible path $\gamma \subseteq B(z,2^{-j})$ and any $w_1,w_2 \in \gamma$ there exists an admissible path $\wt{\gamma}$ from $w_1$ to $w_2$ with $\lmett{\wt{\gamma}} \le M\,\median[2^{-j}]{}$.
\end{itemize}
Consequently, with probability $1$ there is a random $\delta_0 > 0$ such that this holds for all $\delta < \delta_0$ and $z\in \delta^3\Z^2 \cap \D$. Suppose that we are on this event.

Let $\gamma\colon [0,1] \to \Upsilon_\Gamma$ be a geodesic of $\met{\cdot}{\cdot}{\Gamma}$. Let
\[ T = \sup\{ t\in [0,1] : \mett{\gamma(0)}{\gamma(t)}{\Gamma} \le \frac{M}{m}\met{\gamma(0)}{\gamma(t)}{\Gamma} \} . \]
We claim that $T=1$. Indeed, suppose that $T<1$. Then let $\delta < |\gamma(1)-\gamma(T)| \wedge \delta_0$, and let $z \in \delta^3\Z^2$ be the point closest to $\gamma(T)$. Since we are on the event described above, we can find some $\delta'=2^{-j'} \in [\delta^2,\delta]$ that satisfies both items above. Let $u = \inf\{ t \ge T : \gamma(t) \notin B(z,\delta')\}$. Since $|z-\gamma(T)| < \delta^3 < \delta'/10$, the segment $\gamma[T,u]$ crosses the annulus $A(z,\delta'/2,\delta')$, hence
\[ \mett{\gamma(T)}{\gamma(u)}{\Gamma} \le M\median[\delta']{} \le \frac{M}{m}\met{\gamma(T)}{\gamma(u)}{\Gamma} . \]
This implies $\mett{\gamma(0)}{\gamma(u)}{\Gamma} \le \frac{M}{m}\met{\gamma(0)}{\gamma(u)}{\Gamma}$, a contradiction. Therefore we have $T=1$, and we have shown that almost surely
\[ \mett{\gamma(0)}{\gamma(1)}{\Gamma} \le \frac{M}{m}\met{\gamma(0)}{\gamma(1)}{\Gamma} \]
for every $\met{\cdot}{\cdot}{\Gamma}$-geodesic $\gamma$.
\end{proof}

\section{Metrics are equivalent}
\label{sec:metrics_equivalent}

We assume the same setup as in Section~\ref{sec:bi_lip}. The goal of this section is to prove the following result.

\begin{proposition}\label{pr:uniqueness}
Suppose we have the setup above, and assume~\eqref{eq:same_medians_ass}. Then there exists a deterministic constant $c_1 > 0$ such that $\met{\cdot}{\cdot}{\Gamma} = c_1 \mett{\cdot}{\cdot}{\Gamma}$ almost surely. Moreover, $\met{\cdot}{\cdot}{\Gamma}$ is almost surely a measurable function of $\Gamma$.
\end{proposition}

We have shown in Proposition~\ref{pr:bilipschitz} that the two metrics are bi-Lipschitz equivalent. We can therefore make the following assumption.\begin{assumption}
\label{assump:weak_cle_metrics}
Suppose that we have the setup described at the beginning of Section~\ref{se:main_results}. Suppose that $\met{\cdot}{\cdot}{\Gamma}$, $\mett{\cdot}{\cdot}{\Gamma}$ are two weak geodesic \clekp{} metrics coupled with $\Gamma$ and that they are conditionally independent given $\Gamma$. We assume the following
\begin{itemize}
 \item There is a constant $c>1$ such that $c^{-1}\median[\delta]{} \le \mediant[\delta]{} \le c\,\median[\delta]{}$ for all $\delta \in (0,1]$.
 \item Let $c_1$ (resp.\ $c_2$) be the largest (resp.\ smallest) deterministic constant so that $c_1 \met{\cdot}{\cdot}{\Gamma} \leq \mett{\cdot}{\cdot}{\Gamma} \leq c_2 \met{\cdot}{\cdot}{\Gamma}$ almost surely. Assume that $0 < c_1 < c_2 < \infty$.
\end{itemize}
\end{assumption}

Our aim is to show that if Assumption~\ref{assump:weak_cle_metrics} holds then there exists $\epsilon > 0$ so that $\met{\cdot}{\cdot}{\Gamma}$, $\mett{\cdot}{\cdot}{\Gamma}$ satisfy $c_1 \met{\cdot}{\cdot}{\Gamma} \leq \mett{\cdot}{\cdot}{\Gamma} \leq (c_2-\epsilon) \met{\cdot}{\cdot}{\Gamma}$, contradicting the minimality of $c_2$.

The main work will be to show that at many places and on many scales, it is likely that we can find ``shortcuts'' for all crossings of an annulus, meaning that the ratio of the $\mett{\cdot}{\cdot}{\Gamma}$-length to the $\met{\cdot}{\cdot}{\Gamma}$-length is strictly bounded away either from $c_1$ or from $c_2$. We make the following definition which distinguishes the scales on which the ratio is more likely to be bounded away from $c_1$ resp.\ $c_2$. (The reason for using bichordal \slekp{} in the definition below will be explained below Lemma~\ref{le:loopchainreg_event}.)

\newcommand*{\shortcutscales}{\CJ}

Let $\shortcutscales \subseteq \N$ be the set of $j$ such that the following holds. Suppose that $(\eta'_1,\eta'_2)$ is a bichordal \slekp{} in $(2^{-j}\D, -i2^{-j}, i2^{-j})$ as described in Section~\ref{se:bichordal_sle}. Let $I_j$ be the event that $\eta'_1 \cap \eta'_2 \cap B(0,1/2) \neq \varnothing$, and note that by scale-invariance $\p[I_j]$ does not depend on $j$. Then
\begin{equation}\label{eq:shortcut_scale_def}
 \p\left[ \frac{\mett{u_j}{v_j}{\Gamma}}{\met{u_j}{v_j}{\Gamma}} \le \sqrt{c_1 c_2} \mmiddle| I_j \right] \ge 1/2
\end{equation}
where $u_j$ (resp.\ $v_j$) denotes the first (resp.\ last) intersection point of $\eta'_1 \cap \eta'_2 \cap B(0,1/2)$.

The main step in the proof is to show the following. Recall the $\sigma$-algebra $\CF_{z,j}$ and the events $\Esep_{z,j}$ (depending on $\dsep>0$) defined in Section~\ref{se:setup_pf}.

\newcommand*{\Gshortcut}{G^0}

\begin{lemma}
\label{lem:shortcut_event}
Suppose that Assumption~\ref{assump:weak_cle_metrics} holds. Let~$\shortcutscales$ be as defined above~\eqref{eq:shortcut_scale_def}, and suppose additionally that the number of $n \in \N$ such that
\begin{equation}\label{eq:shortcut_density}
\frac{\abs{\shortcutscales \cap \{ n,\ldots,2n \}}}{n} \ge \frac{1}{2}
\end{equation}
is infinite.
For any $p\in (0,1)$ and $\dsep>0$ there exists $\epsilon > 0$ such that the following is true.
Let $\Gshortcut_{z,j}$ be the event that the following holds. For every admissible path $\gamma$ crossing the annulus $A_{z,j}$ there exist $u,v \in \gamma$ such that if $\gamma' \subseteq \gamma$ is the segment from $u$ to $v$, then
\[
\mettres{A_{z,j}}{u}{v}{\Gamma} \le (c_2-\epsilon)\lmet{\gamma'}
\quad\text{and}\quad
\lmet{\gamma'} \ge \epsilon\,\median[2^{-j}]{}
. \]
For every $0<\delta_0\le 1$ there exists $0 < \delta \le \delta_0$ such that for each $z$ with $B(z,\delta) \subseteq \D$, at least a $1/4$ fraction of $j=\log_2\delta^{-1},\ldots,2\log_2\delta^{-1}$ satisfy
\[ \p[ \Gshortcut_{z,j} \mid \CF_{z,j} ] \one_{\Esep_{z,j}}  \geq p \one_{\Esep_{z,j}} . \]
\end{lemma}

We begin by defining a notion of ``good'' annulus in which it is easy to find shortcuts.

\newcommand*{\Eloopchainreg}{E^2}
\newcommand*{\distloopchainreg}{r_2}

\begin{figure}[ht]
\centering
\includegraphics[page=1,width=0.45\textwidth]{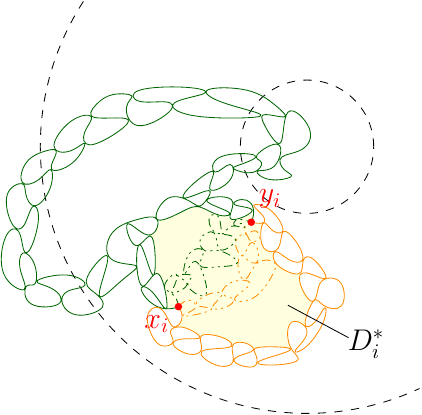}
\hspace{0.05\textwidth}
\includegraphics[page=2,width=0.45\textwidth]{chain_regions}
\caption{Illustration of the regions between the intersecting loops in the definition of the event $\Eloopchainreg_{z,j}$. On the left, $D_i^*$ is shown in yellow and on the right $U_i$ is shown in purple.}
\label{fi:regions_between_chain}
\end{figure}

For $n \in \N$, $\distloopchainreg > 0$, let $\Eloopchainreg_{z,j}$ be the event that the following hold.
\begin{itemize}
\item There is a collection of at most $n$ points $u_i \in A_{z,j}$ that separate $\bIn A_{z,j}$ from $\bOut A_{z,j}$ in $A_{z,j} \cap \Upsilon_{\Gamma}$.
\item For each $i$, let $\CL_{i,1},\CL_{i,2} \in \Gamma$ be the loop(s) passing through $u_i$ (where $\CL_{i,1}=\CL_{i,2}$ is possible). There are two non-overlapping strands $\ell_{i,1} \subseteq \CL_{i,1}$, $\ell_{i,2} \subseteq \CL_{i,2}$ and further intersection points $v_i,x_i,y_i \in \ell_{i,1} \cap \ell_{i,2}$ such that
\begin{itemize}
\item the strands $\ell_{i,1},\ell_{i,2}$ start and end at $x_i,y_i$, respectively, and pass through $u_i,v_i$,
\item $\abs{u_i-z}-\abs{v_i-z} \ge \distloopchainreg 2^{-j}$,
\item if $U_i$ denotes the regions bounded between the two segments from $u_i$ to $v_i$, and $D^*_i$ denotes the connected component of $\C \setminus ((\CL_{i,1} \cup \CL_{i,2}) \setminus (\ell_{i,1} \cup \ell_{i,2}))$ containing $U_i$, then $D^*_i \subseteq A_{z,j}$ and $\distE(U_i,\partial D^*_i) \ge \distloopchainreg 2^{-j}$.
\end{itemize}
\end{itemize}
See Figure~\ref{fi:regions_between_chain} for an illustration of the second bullet point. Note that it is allowed that $B(z,2^{-j-1}) \cap \Upsilon_\Gamma = \varnothing$ in which case the collection $\{u_i\}$ is empty.

\begin{lemma}\label{le:loopchainreg_event}
For any $b>1$ there exist $n\in\N$, $\distloopchainreg>0$, and $c>0$ such that for each $z,j_0$ with $B(z,2^{-j_0}) \subseteq \D$ and $K \in \N$, the probability that less than a $9/10$ fraction of $\Eloopchainreg_{z,j}$ where $j=j_0,\ldots,j_0+K$ occur is at most $ce^{-bK}$.
\end{lemma}

\begin{proof}
 By the independence across scales proved in \cite[Proposition~4.9]{amy-cle-resampling}, the statement follows once we argue that for a given annulus $A_{z,j}$, the probability of $\Eloopchainreg_{z,j}$ can be made as close to $1$ as we like by choosing the parameters suitably. We now argue the latter. The existence of a finite separating set as required in the first bullet point follows from \cite[Lemma~5.14]{amy-cle-resampling}. To see that we can choose the points $x_i,u_i,v_i,y_i$ as required in the second bullet point, note that when two strands of (one or two) loops intersect, they intersect at infinitely many points, and their outer boundaries do not intersect.
\end{proof}

Consider an annulus $A_{z,j}$ where $\Eloopchainreg_{z,j}$ occurs. Our goal is to find shortcuts in each of the regions $U_i$, so that there is a shortcut for each geodesic crossing the annulus $A_{z,j}$. Recall that the conditional law of a strand of a CLE loop given its remainder is given by a chordal SLE. In particular, for a pair of intersecting CLE loops, the conditional law of the pair of intersecting strands given their remainders is given by a bichordal SLE as described in Section~\ref{se:bichordal_sle}. We make this precise below.

\newcommand*{\Gexpl}{G^2}
\newcommand*{\nexpl}{n_2}

\begin{figure}[ht]
\centering
\includegraphics[width=0.6\textwidth]{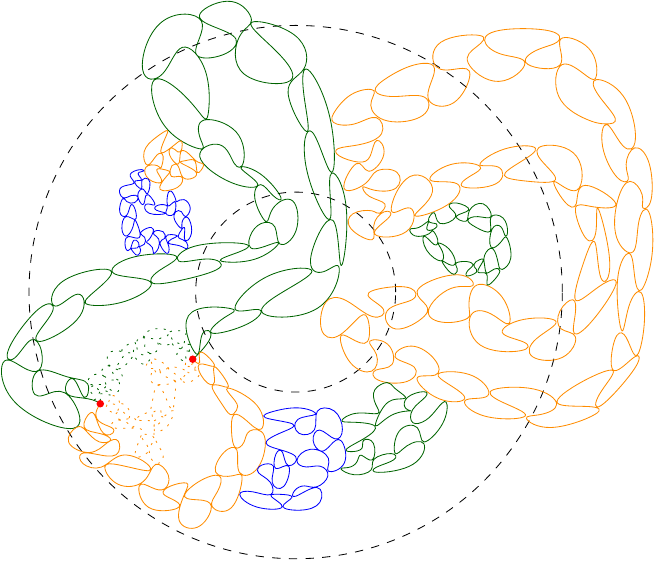}
\caption{The conditional law of the pair of dotted strands is a bichordal \slekp{}.}
\end{figure}

Consider the following countable collection of partial explorations of $\Gamma$. Let $W,V$ be open, simply connected subsets with $W \Subset V \Subset A_{z,j}$ such that $\partial W, \partial V$ are polygonal paths with rational vertices. Let $\Gamma_\outside^{*,V,W}$ be the partially explored \clekp{} defined in Section~\ref{se:mcle}. Suppose that we are on the event that there are exactly $4$ marked points in $\Gamma_\outside^{*,V,W}$, and let us further condition on a link pattern of the two unfinished strands $(\eta'_1,\eta'_2)$. Let $x$ (resp.\ $y$) be the first (resp.\ last) points where $\eta'_1,\eta'_2$ intersect, and let $D^*$ denote the connected component containing the remainder of $(\eta'_1,\eta'_2)$. Then, by Proposition~\ref{pr:bichordal_sle}, the conditional law of the remainder of $(\eta'_1,\eta'_2)$ is that of a bichordal \slekp{} in $(D^*,x,y)$.

Let us enumerate the collection of explorations described above. Suppose we are on the event $\Eloopchainreg_{z,j}$, and let $U_i$ be one of the regions in the definition of the event. We have that for at least one of the explorations, the strands adjacent to $U_i$ are exactly the remaining strands $(\eta'_1,\eta'_2)$ from the exploration. For $\nexpl \in \N$, we consider the first $\nexpl$ explorations in the enumeration. Let $\Gexpl_{z,j}$ be the event that every $U_i$ in the definition of $\Eloopchainreg_{z,j}$ is between the remaining strands for one of the first $\nexpl$ explorations. By choosing $\nexpl$ large, we can make the probability of $\Gexpl_{z,j}$ as close to $1$ as we want. By the (local) total variation continuity of the multichordal \clekp{} law (Proposition~\ref{pr:mccle_gasket_tv_convergence}), we can choose $\nexpl \in \N$ so that this holds uniformly for all $z,j$ which we record in the following lemma.
\begin{lemma}\label{le:expl_trials}
 For each $\dsep>0$ and $q>0$ there exists $\nexpl \in \N$ such that
 \[ \p[(\Gexpl_{z,j})^c \cap \Eloopchainreg_{z,j} \mid \wt{\CF}_{z,j}] \one_{\Esep_{z,j}} \le q \]
 for each $z,j$.
\end{lemma}

Now, suppose that $c_1 < c_2$ in Assumption~\ref{assump:weak_cle_metrics}. Consider a scale $j \in \shortcutscales$ as defined above~\eqref{eq:shortcut_scale_def}. Then by definition, a shortcut exists with positive probability in $2^{-j}\D$. In order to prove Lemma~\ref{lem:shortcut_event}, we need to make the probability as high as we want and find shortcuts in each of the regions $U_i$ in the definition of the event $\Eloopchainreg_{z,j}$. This will be done by repeatedly trying in many disjoint areas of the annulus and using the spatial independence (at the cost of changing the scale and the value of $\epsilon>0$ by a fixed factor). To achieve this, we will phrase the positive probability event in~\eqref{eq:shortcut_scale_def} in terms of flow line intersections. Recall from Section~\ref{se:bichordal_sle} that we can construct the region bounded between $\eta'_1,\eta'_2$ by flow lines $\eta_1,\eta_2$ with angle difference $\angledouble$ from $-i2^{-j}$ to $i2^{-j}$.

\newcommand*{\Gshortcutfl}{G^3}
\newcommand*{\pshortcutfl}{p_3}

\begin{lemma}\label{le:shortcut_flow_line}
There exist $\epsilon > 0$, $\pshortcutfl>0$ such that the following is true. Suppose that $j \in \shortcutscales$, let $h$ be a GFF in $2^{-j+1}\D$, and $w_1,w_2 \in 2^{-j}\D$ sampled independently according to Lebesgue measure. Let $\wt{\eta}_1$, $\wt{\eta}_2$ be flow lines of $h$ with angles $\pm\theta = \pm (\kappa-2)\lambda/(2\chi)$ starting at $w_1$ (resp.\ $w_2$) and stopped upon exiting $2^{-j}\D$. Let $\wt{\Gamma}$ be a conditionally independent CLE$_{\kappa'}$ in each of the components bounded between $\wt{\eta}_1,\wt{\eta}_2$, also coupled with $h$.

Let $\Gshortcutfl_j$ be the event that the right side of $\wt{\eta}_1$ intersects the left side of $\wt{\eta}_2$ with angle difference $2\theta$ and there exist some intersection points $\wt{u},\wt{v} \in \wt{\eta}_1\cap\wt{\eta}_2$ such that if $\wt{U}$ denotes the regions bounded between the segments of $\wt{\eta}_1$, $\wt{\eta}_2$ from $\wt{u}$ to $\wt{v}$, then
\[
\mettres{\ol{\wt{U}}}{\wt{u}}{\wt{v}}{\wt{\Gamma}} \le \sqrt{c_1 c_2}\,\metres{\ol{\wt{U}}}{\wt{u}}{\wt{v}}{\wt{\Gamma}}
\quad\text{and}\quad
\metres{\ol{\wt{U}}}{\wt{u}}{\wt{v}}{\wt{\Gamma}} \ge \epsilon\,\median[2^{-j}]{}
. \]
Then $\p[\Gshortcutfl_j] \ge \pshortcutfl$.
\end{lemma}

\begin{proof}
Suppose that $(\eta'_1,\eta'_2)$ are coupled with a GFF in $2^{-j}\D$ as described in Section~\ref{se:bichordal_sle}. Let
\[ p_0 = \p[I_j] . \]
Let $s_1>0$ and let $G^{1,s_1}$ be the event that both the right boundary of $\eta'_1$ and the left boundary of $\eta'_2$ between $u_j$ and $v_j$ are contained in $B(0,1-s_1)$. By scaling invariance, $\p[G^{1,s_1}]$ does not depend on $j$, and we can choose $s_1$ so that $\p[(G^{1,s_1})^c] < p_0/8$. Next, let $G^{2,s_2}$ be the event that each $\eta'_1,\eta'_2$ disconnect a ball of radius $s_2 2^{-j}$ within traveling distance $2^{-j}/4$ after visiting $u_j$. Again, $\p[G^{2,s_2}]$ does not depend on $j$, and we can choose $s_2$ so that $\p[(G^{2,s_2})^c] < p_0/8$. Finally, by Lemma~\ref{le:len_lb}, we can choose $\epsilon > 0$ such that
\[
 \p[ \met{u_j}{v_j}{\Gamma} < \epsilon\,\median[2^{-j}]{} ] \le p_0/8 .
\]
We then have
\[
 \p\left[ \frac{\mett{u_j}{v_j}{\Gamma}}{\met{u_j}{v_j}{\Gamma}} \le \sqrt{c_1 c_2} ,\ \met{u_j}{v_j}{\Gamma} \ge \epsilon\,\median[2^{-j}]{} ,\, I_j ,\ G^{1,s_1} ,\, G^{2,s_2} \right] \ge p_0/8
\]
Now if we sample $w_1,w_2 \in 2^{-j}\D$ so that they are within the two respective balls as in the event $G^{2,s_2}$, then the flow lines starting from $w_1$ (resp.\ $w_2$) will merge with $\eta_1$ (resp.\ $\eta_2$) before they trace the segments from $u_j$ to $v_j$. Therefore the event for the flow lines occurs before they get within distance $s_1 2^{-j}$ to $\partial \D$. Hence, by absolute continuity there exists $\pshortcutfl>0$ such that the same is true for the flow lines $\wt{\eta}_1,\wt{\eta}_2$ of a GFF in $2^{-j+1}\D$.
\end{proof}

\newcommand*{\ntrials}{n_3}

As mentioned earlier, we argue that if a shortcut occurs with positive probability, then by spacial independence it occurs somewhere with high probability.

\begin{lemma}\label{le:shortcut_loop_intersection}
For every $p\in (0,1)$ there exist $\ntrials\in\N$ and $\epsilon > 0$ such that the following is true. Suppose that $D^*$ is a simply connected domain, let $\varphi\colon \h \to D^*$ be a conformal transformation, and let $x = \varphi(0)$, $y = \varphi(\infty)$. Let $s_1 > s_2 > 0$ be such that $s_1 = \distE(\varphi(i), \partial D^*)$, and suppose that
\[ \abs{\{j \in \shortcutscales : s_2 < 2^{-j} < s_1 \}} \ge \ntrials . \]
Let $(\eta'_1,\eta'_2)$ be a bichordal SLE$_{\kappa'}$ in $(D^*,x,y)$ and $\Gamma$ a conditionally independent CLE$_{\kappa'}$ in each of the complementary components. Let $G_D^*$ denote the event that there exist $u,v \in \eta'_1\cap\eta'_2$ such that
\[
\mett{u}{v}{\Gamma} \le \sqrt{c_1 c_2}\,\met{u}{v}{\Gamma}
\quad\text{and}\quad
\met{u}{v}{\Gamma} \ge \epsilon\,\median[s_2]{}
. \]
Then $\p[G_D^*] \ge p$.
\end{lemma}

\begin{proof}
Let $h$ be a GFF in $\h$ with boundary values $2\lambda'-\pi\chi = (\kappa-2)\lambda$ (resp.\ $-2\lambda'+\pi\chi = -(\kappa-2)\lambda$) on $\R_+$ (resp.\ $\R_-$). Let $\CG_m$ denote the $\sigma$-algebra generated by the restriction of $h$ to $B(0,2^m)\cap\h$. We use the independence across scales property of the GFF, see \cite[Section~4.1.6, boundary case]{amy-cle-resampling} where the notion of $M$-good scales is defined.

Let $\CK = \{ m\in\Z : s_2 < 2^{-j} < \distE(\varphi(i2^m), \partial D^*)/100 \le 2^{-j+1} < s_1 \text{ for some } j \in \shortcutscales \}$. By our assumption, we have $\abs{\CK} \ge \ntrials$. By \cite[Lemma~4.2]{amy-cle-resampling}, there exists $M>0$ such that the probability that less than $1/2$ fraction of $m\in\CK$ are $M$-good for $h$ is $O(e^{-|\CK|}) \le O(e^{-\ntrials})$. Let $\CK' \subseteq \CK$ denote the $M$-good scales. Pick $\ntrials$ large enough so that
\begin{equation}\label{eq:bichordal_good_scales}
\p[|\CK'| \ge \ntrials/2] \ge 1-(1-p)/2 .
\end{equation}

Let $m\in\CK$ and $j = \lceil -\log_2(\distE(\varphi(i2^m), \partial D^*)) \rceil$. Let $h_{D^*} = h\circ\varphi^{-1}-\chi\arg(\varphi^{-1})'$. By Koebe's $1/4$-theorem, $\varphi(B(i2^m,2^m/16))$ contains a ball of radius $2^{-j}/64$. Let $G^{*,1}_{m}$ denote the event that the event from Lemma~\ref{le:shortcut_flow_line} occurs for the restriction of $h_{D^*}$ to $B(\varphi(i2^m),2^{-j}/100)$. By translation invariance and absolute continuity, there is $\pshortcutfl'>0$ (depending only on $\pshortcutfl$) such that $\p[G^{*,1}_{m}] \ge \pshortcutfl'$.

Recall the angle $\theta = (\kappa-2)\lambda/(2\chi)$ defined above, and let $\wt{\eta}_1,\wt{\eta}_2$ be the flow lines of $h_{D^*}$ in the event $G^{*,1}_{m}$. For some $p_4 > 0$, let $G^{*,2}_m$ denote the event that if we sample the points $w_1,w_2 \in B(\varphi(i2^m),2^{-j}/100)$ as in the definition of the event $G^{*,1}_{m}$, and additionally consider the flow lines $\wt{\eta}_{1,L},\wt{\eta}_{2,R}$ with angles $\theta+\pi$ (resp.\ $-\theta-\pi$) starting from $w_1$ (resp.\ $w_2$), then the conditional probability that $\wt{\eta}_{1,L}$ (resp.\ $\wt{\eta}_{2,R}$) exits $B(\varphi(i2^m),2^{-j}/2)$ without intersecting $\wt{\eta}_{2},\wt{\eta}_{2,R}$ (resp.\ $\wt{\eta}_{1},\wt{\eta}_{1,L}$) is at least $p_4$. By scale invariance, we can choose $p_4 > 0$ such that $\p[G^{*,2}_m] > 1-\pshortcutfl'/2$ for every $m$.

Next, for some $p_5>0$, let $G^{*,3}_m$ denote the event that for every pair of flow lines $\eta_{\theta+\pi},\eta_{-\theta-\pi}$ of $h$ with angles $\theta+\pi$ (resp.\ $-\theta-\pi$) crossing out of the annulus $A(i2^m,2^m/16,2^m/8)$, the conditional probability given $h\big|_{B(i2^m,2^m/8)}$ that $\eta_{\theta+\pi}$ hits $\R_-$ and $\eta_{-\theta-\pi}$ hits $\R_+$ before exiting $A(0,2^{m-1},2^{m+1})\cap\h$ and without intersecting each other is at least $p_5$. By scale invariance, $\p[G^{*,3}_m]$ does not depend on $m$, and we can choose $p_5>0$ such that $\p[G^{*,3}_m] > 1-p_4\pshortcutfl'/4$.

Now, let $G^{*,4}_{m}$ denote the event that $G^{*,1}_{m} \cap G^{*,2}_m \cap G^{*,3}_m$ occurs and that the two flow lines $\wt{\eta}_{1,L},\wt{\eta}_{2,R}$ hit the respective boundary parts before exiting $\varphi(A(0,2^{m-1},2^{m+1})\cap\h)$. This event is measurable with respect to $h\big|_{A(0,2^{m-1},2^{m+1}) \cap \h}$ and the internal metrics within $A(0,2^{m-1},2^{m+1}) \cap \h$. By the choices of the parameters, we have
\[ \p[G^{*,4}_{m}] \ge p_5 p_4 \pshortcutfl'/4 \]
for every $m\in\CK$. If additionally $m$ is an $M$-good scale for $h$, then
\[ \p[G^{*,4}_{m} \mid \CG_{m-1}]\one_{m\in\CK'} \ge p_6\one_{m\in\CK'} \]
where $p_6>0$ depends on $\pshortcutfl',p_4,p_5,M$.

By picking $\ntrials$ larger if necessary, we can assume that
\[ (1-p_6)^{\ntrials/2} < (1-p)/2 . \]
Combining this with~\eqref{eq:bichordal_good_scales}, we conclude
\[ \p\!\left[ \bigcup_m G^{*,4}_{m} \right] \ge p . \]
To obtain the conclusion of the lemma, observe that if $G^{*,4}_{m}$ occurs, then necessarily $\eta'_1$ and $\eta'_2$ merge into the flow lines involved, so that the region from the event $G^{*,1}_{m}$ is as desired.
\end{proof}

\begin{proof}[Proof of Lemma~\ref{lem:shortcut_event}]
Let $\dsep>0$, $p \in (0,1)$ be given, and let $q=(1-p)/2$. By Lemmas~\ref{le:loopchainreg_event} and~\ref{le:expl_trials}, there exist $\distloopchainreg>0$, $\nexpl \in \N$ such that
\[ \p[(\Gexpl_{z,j} \cap \Eloopchainreg_{z,j})^c \mid \wt{\CF}_{z,j}] \one_{\Esep_{z,j}} \le q \]
for each $z,j$.

Consider the explorations described above Lemma~\ref{le:expl_trials}. On the event $\Gexpl_{z,j} \cap \Eloopchainreg_{z,j}$, we need to consider at most $\nexpl$ of them. Apply Lemma~\ref{le:shortcut_loop_intersection} with $1-q/\nexpl$ in place of $p$, and let $\ntrials \in \N$ be as in the statement of Lemma~\ref{le:shortcut_loop_intersection}.

Note that if $n \in \N$ is sufficiently large and~\eqref{eq:shortcut_density} holds, then the number of $j \in \{n,\ldots,2n\}$ with
\[ \abs{\shortcutscales \cap \{j,\ldots,j+5\ntrials\}} \ge \ntrials \]
is at least $n/4$.

Let $j \in \N$ be such that $\abs{\{ j' \in \shortcutscales : \distloopchainreg 2^{-j-5\ntrials} \le 2^{-j'} \le \distloopchainreg 2^{-j} \} } \ge \ntrials$. We apply Lemma~\ref{le:shortcut_loop_intersection} to each of the regions $D^*$ arising from the $\nexpl$ explorations. By a union bound we get that with conditional probability at least $1-2q = p$ the statement of Lemma~\ref{lem:shortcut_event} holds with $\median[\distloopchainreg 2^{-j-5\ntrials}]{}$ in place of $\median[2^{-j}]{}$. By Proposition~\ref{pr:quantiles_metric}, we have $\median[\distloopchainreg 2^{-j-5\ntrials}]{} \asymp \median[2^{-j}]{}$ with an implicit constant depending only on $\distloopchainreg, \ntrials$ (and therefore on $p,\dsep$).
\end{proof}

\begin{proof}[Proof of Proposition~\ref{pr:uniqueness}]
Suppose that $c_1<c_2$ in Assumption~\ref{assump:weak_cle_metrics}, and suppose that the number of $n \in \N$ such that~\eqref{eq:shortcut_density} holds is infinite. We argue that there is a deterministic constant $\epsilon' > 0$ and a sequence $(\delta_n)$ with $\delta_n \searrow 0$ such that the following holds. For every $n$ and $z \in \delta_n^3 \Z^2 \cap\D$ there exists some random $2^{-j'} \in [\delta_n^2,\delta_n]$ such that for every $\met{\cdot}{\cdot}{\Gamma}$-geodesic $\gamma \subseteq B(z,2^{-j'})$ that crosses the annulus $A_{z,j'}$ we have $\mett{\gamma(0)}{\gamma(1)}{\Gamma} \le (c_2-\epsilon')\lmet{\gamma}$. By the same argument as in the proof of Proposition~\ref{pr:bilipschitz}, this implies that $\mett{\cdot}{\cdot}{\Gamma} \le (c_2-\epsilon')\met{\cdot}{\cdot}{\Gamma}$, contradicting the minimality of $c_2$. On the other hand, if the number of $n \in \N$ such that~\eqref{eq:shortcut_density} holds is finite, swapping the roles of $\met{}{}{},\mett{}{}{}$ shows that $\met{\cdot}{\cdot}{\Gamma} \le (c_1^{-1}-\epsilon')\mett{\cdot}{\cdot}{\Gamma}$, contradicting the maximality of $c_1$.

Let $\delta > 0$ be such that the conclusion of Lemma~\ref{lem:shortcut_event} holds. By Proposition~\ref{pr:ball_crossing} and the discussion after, applied to the event $\Gshortcut_{z,j}$ defined in Lemma~\ref{lem:shortcut_event}, we see that there are constants $M>1$, $\epsilon>0$ such that the following holds. Let $G^{\delta,z}$ be the event that there exists a (random) scale $j' \in \{\log_2\delta^{-1},\ldots,2\log_2\delta^{-1}\}$ such that for every $\met{\cdot}{\cdot}{\Gamma}$-geodesic $\gamma \subseteq B(z,2^{-j'})$ that crosses the annulus $A_{z,j'}$ we have
\begin{itemize}
\item $\lmet{\gamma} \le M\median[2^{-j'}]{}$, and
\item there are points $u,v \in \gamma$ with $\mett{u}{v}{\Gamma} \le (c_2-\epsilon)\met{u}{v}{\Gamma}$ and $\met{u}{v}{\Gamma} \ge \epsilon\,\median[2^{-j'}]{}$.
\end{itemize}
Then $\p[(G^{\delta,z})^c] = O(\delta^8)$.

Suppose that we are on the event $G^{\delta,z}$ and $j'$ is as described above. Let $\gamma \subseteq B(z,2^{-j'})$ be a $\met{\cdot}{\cdot}{\Gamma}$-geodesic that crosses the annulus $A_{z,j'}$. Let $\gamma' \subseteq \gamma$ be the subsegment between $u,v$. Then
\[ \mett{u}{v}{\Gamma} \le (c_2-\epsilon)\lmet{\gamma'} . \]
Moreover, $\lmet{\gamma'} \ge \epsilon\,\median[2^{-j'}]{}$, and hence $\lmet{\gamma'} \ge \frac{\epsilon}{M}\lmet{\gamma}$. Further,
\[ \lmett{\gamma\setminus\gamma'} \le c_2\lmet{\gamma\setminus\gamma'} \]
due to Assumption~\ref{assump:weak_cle_metrics}. Together, this implies
\[ \mett{\gamma(0)}{\gamma(1)}{\Gamma} \le \left(c_2-\frac{\epsilon}{M}\epsilon\right)\lmet{\gamma} , \]
concluding the proof.
\end{proof}

\section{Uniqueness of the exponent}
\label{se:exponent}

In this section we prove the scaling covariance and the uniqueness of the exponent. This will complete the proofs of Theorems~\ref{thm:uniqueness} and~\ref{thm:exponent}.

Consider the same setup as described at the beginning of Section~\ref{se:main_results}. In particular, we let $\Gamma_\D$ be a nested $\CLE_{\kappa'}$ in $\D$, let $\CL$ be the outermost loop of $\Gamma_\D$ that surrounds~$0$, and \emph{condition on the event that $\CL \subseteq \D$}. Let $C$ be the regions surrounded by $\CL$, let $\Gamma_C$ be the loops of $\Gamma_\D$ contained in $\ol{C}$, and let $\Gamma = \{\CL\} \cup \Gamma_C$.

The proofs will be completed upon showing the following lemmas. Recall (see e.g.\ the proof of Lemma~\ref{lem:conf_map_abs_cont} for details) that for $\lambda \in (0,1]$, the law of $\lambda\Gamma$ is absolutely continuous with respect to the law of $\Gamma$ conditioned on the event that $\CL \subseteq \lambda\D$.

\begin{lemma}\label{le:scaling_covariance}
Suppose we have the setup described at the beginning of Section~\ref{se:main_results}. Let $\met{\cdot}{\cdot}{\Gamma}$ be a weak geodesic \clekp{} metric coupled with~$\Gamma$. Then there exists a constant $\alpha > 0$ such that for each $\lambda \in (0,1]$, almost surely $\met{\lambda\cdot}{\lambda\cdot}{\lambda\Gamma} = \lambda^\alpha \met{\cdot}{\cdot}{\Gamma}$.
\end{lemma}

\begin{lemma}\label{le:exponent_unique}
Suppose we have the setup described at the beginning of Section~\ref{se:main_results}. Let $\met{\cdot}{\cdot}{\Gamma}, \mett{\cdot}{\cdot}{\Gamma}$ be two geodesic \clekp{} metrics, and let $\alpha, \wt{\alpha}$ be the respective exponents in Lemma~\ref{le:scaling_covariance}. Then $\alpha = \wt{\alpha}$.
\end{lemma}

In particular, this shows that the result of Proposition~\ref{pr:uniqueness} holds also without requiring a priori that $(\median[\delta]{})$ and $(\mediant[\delta]{})$ are comparable~\eqref{eq:same_medians_ass}. This completes the proof of Theorem~\ref{thm:uniqueness}.

\begin{lemma}\label{le:exponent_bounds}
In the setup of Lemma~\ref{le:scaling_covariance} have $\alpha \in [\ddouble \vee 1, d_\SLE]$.
\end{lemma}

\begin{proof}[Proof of Lemma~\ref{le:scaling_covariance}]
Suppose that we have the setup described above, and $\met{\cdot}{\cdot}{\Gamma}$ is a weak geodesic \clekp{} metric. For $\lambda \in (0,1]$, let $\mett{\cdot}{\cdot}{\Gamma} = \met{\lambda\cdot}{\lambda\cdot}{\lambda\Gamma}$. It is easy to see that $\mett{\cdot}{\cdot}{\Gamma}$ is again a weak geodesic \clekp{} metric, and $\mediant[\delta]{} = \median[\lambda\delta]{}$ for $\delta \in (0,1]$ where $\mediant[\delta]{}$ is defined in Section~\ref{se:setup_pf}. In particular, $\mediant[\delta]{}/\median[\delta]{} \in [\lambda^{d_\SLE+o(1)},\lambda^{\ddouble+o(1)}]$ is bounded uniformly in $\delta$ for fixed $\lambda$ by Proposition~\ref{pr:quantiles_metric}. Therefore the condition~\eqref{eq:same_medians_ass} of Proposition~\ref{pr:uniqueness} is satisfied, implying that $\mett{\cdot}{\cdot}{\Gamma} = c(\lambda)\met{\cdot}{\cdot}{\Gamma}$ almost surely for a constant $c(\lambda) > 0$.

Moreover, we have $c(\lambda \lambda') = c(\lambda)c(\lambda')$ since $\met{\lambda \lambda' \cdot}{\lambda \lambda' \cdot}{\lambda \lambda' \Gamma} = c(\lambda) \met{\lambda' \cdot}{\lambda' \cdot}{\lambda' \Gamma} = c(\lambda)c(\lambda')\met{\cdot}{\cdot}{\Gamma}$. Also, $\lambda \mapsto c(\lambda)$ is monotone since otherwise there would exist a sequence of $\lambda_n \searrow 0$ with $c(\lambda_n) \to \infty$ which would contradict the continuity of $\met{\cdot}{\cdot}{\Gamma}$. Hence, there exists $\alpha>0$ such that $c(\lambda) = \lambda^\alpha$ for all $\lambda \in (0,1]$.
\end{proof}

\begin{proof}[Proof of Lemma~\ref{le:exponent_unique}]
Suppose $\met{\cdot}{\cdot}{\Gamma}, \mett{\cdot}{\cdot}{\Gamma}$ are two geodesic \clekp{} metrics with $\median[r]{} = r^\alpha$ and $\mediant[r]{} = r^{\wt{\alpha}}$, and suppose that $\wt{\alpha} > \alpha$. We claim that this would imply $\mett{\cdot}{\cdot}{\Gamma} = 0$. This will follow again by the argument used in the proof of Proposition~\ref{pr:bilipschitz}. Indeed, apply Proposition~\ref{pr:ball_crossing} to $\mett{\cdot}{\cdot}{\Gamma}$ and the discussion that follows with the events $G_{z,j} = \{ d_{z,j} \ge m\,\median[2^{-j}]{} \}$ where $d_{z,j}$ is as in Lemma~\ref{le:len_lb_conditional}. Then, for suitable choices of $m,M > 0$, for each $\delta > 0$ it holds with probability $1-O(\delta^8)$ that there exists $j_1 \in \{ \log_2(\delta^{-1}),\ldots,\log_2(\delta^{-2}) \}$ such that for any $\met{\cdot}{\cdot}{\Gamma}$-geodesic $\gamma \subseteq B(z,2^{-j_1})$ that crosses $A_{z,j_1}$ we have
\begin{equation}
\label{eqn:an_crossings_bound}
\mett{\gamma(0)}{\gamma(1)}{\Gamma} \le M\mediant[2^{-j_1}]{} \ll m\,\median[2^{-j_1}]{} \le \lmet{\gamma}.
\end{equation}
By the same argument as in the proof of Proposition~\ref{pr:bilipschitz}, the inequality~\eqref{eqn:an_crossings_bound} implies that
\[ \mett{\cdot}{\cdot}{\Gamma} \le \frac{M}{m}\frac{\mediant[2^{-j_1}]{}}{\median[2^{-j_1}]{}} \met{\cdot}{\cdot}{\Gamma}\]
for some $j_1$ that can be made arbitrarily large, implying $\mett{\cdot}{\cdot}{\Gamma} = 0$.
\end{proof}

\begin{proof}[Proof of Lemma~\ref{le:exponent_bounds}]
Proposition~\ref{pr:quantiles_metric} implies that $\alpha \in [\ddouble, d_\SLE]$. We argue that $\alpha \ge 1$ which is a consequence of the geodesic property. Indeed, applying Lemma~\ref{le:len_lb_conditional} as before, we see that for some constant $m>0$, there almost surely exists $j_1$ that can be made arbitrarily large and such that $\lmet{\gamma} \ge m\,\median[2^{-j_1}]{}$ for every path $\gamma$ crossing $A_{z,j_1}$ for some $z$. Since continuous paths have dimension at least $1$, every geodesic crosses at least $c2^{j_1}$ such annuli. Therefore we must have $\median[2^{-j_1}]{} \lesssim 2^{-j_1}$, otherwise the geodesics would have infinite length.
\end{proof}

\section{General domains and conformal covariance}
\label{se:general_domains}

In this section we construct the geodesic \clekp{} metric on interior clusters in general domains and in the whole-plane, and prove their conformal covariance.

\subsection{Conformal covariance}
\label{sec:conformally_covariant}

In this subsection we prove the conformal covariance in the setup of Theorem~\ref{thm:conformal_covariance}. We will prove a slightly more general version which is stated as Theorem~\ref{thm:conformal_covariance_gen}. We begin by showing rotational invariance.

\begin{lemma}\label{le:rotational_invariance}
Suppose we have the setup described at the beginning of Section~\ref{se:main_results}. Let $\met{\cdot}{\cdot}{\Gamma}$ be a geodesic \clekp{} metric, and $\theta\in\R$. Then
\[ \met{e^{i\theta}\cdot}{e^{i\theta}\cdot}{e^{i\theta}\Gamma} = \met{\cdot}{\cdot}{\Gamma} \]
almost surely.
\end{lemma}

\begin{proof}
Let $\mett{\cdot}{\cdot}{\Gamma} = \met{e^{i\theta}\cdot}{e^{i\theta}\cdot}{e^{i\theta}\Gamma}$. By Theorem~\ref{thm:uniqueness}, we have $\mett{\cdot}{\cdot}{\Gamma} = c\,\met{\cdot}{\cdot}{\Gamma}$ for some constant $c$. By considering a rotationally invariant (in law) observable, e.g.\ the $\met{\cdot}{\cdot}{\Gamma}$-diameter of~$\Upsilon_\Gamma$, we conclude $c=1$.
\end{proof}

\begin{lemma}
\label{lem:conf_map_abs_cont}
Consider the setup of Theorem~\ref{thm:conformal_covariance}, i.e.\ let $\varphi \colon U \to \wt{U}$ be a conformal transformation between two simply connected regions $U,\wt{U} \subseteq \D$.  Then for each $U_1 \Subset U$, the law of $\varphi(U_1 \cap \Upsilon_\Gamma)$ is absolutely continuous with respect to the law of $\varphi(U_1) \cap \Upsilon_\Gamma$.
\end{lemma}

In the setting of Lemma~\ref{lem:conf_map_abs_cont}, let
\[ \mettres{U}{\cdot}{\cdot}{\Gamma} \defeq \metres{\wt{U}}{\varphi(\cdot)}{\varphi(\cdot)}{\varphi(\Gamma)}.\]
We note that it follows from Lemma~\ref{lem:conf_map_abs_cont} that this is well-defined by absolute continuity. Indeed, the metric $\mettres{U_1}{\cdot}{\cdot}{\Gamma}$ is well-defined for each $U_1 \Subset U$, and by Lemma~\ref{le:internal_metrics_compatibility} it extends to a metric on all of $U \cap \Upsilon_\Gamma$.

\begin{proof}[Proof of Lemma~\ref{lem:conf_map_abs_cont}]
We perform a partial resampling procedure of the \clekp{} as described in \cite[Section~5]{amy-cle-resampling}. We select regions $(A,B) \in \domainpair{\D}$ with $B \subseteq \D \setminus U_1$ in some random fashion, and resample the part $\Gamma_\inside^{*,B,A}$ according to its conditional law given $\Gamma_\outside^{*,B,A}$. We repeat this procedure a random number of times, yielding a coupling of $\Gamma$ with $\Gamma^\resampled$ whose law is the same as that of $\Gamma$. As a consequence of Proposition~\ref{pr:link_probability} (see \cite[Section~5.4]{amy-cle-resampling} for details), the probability is positive that the gasket $U_1 \cap \Upsilon_\Gamma$ is unchanged and $\Gamma^\resampled$ contains a loop $\CL^\resampled$ that disconnects a region $V$ with $U_1 \subseteq V \subseteq U$ from $\partial\D$. Consider the exploration of $\Gamma^\resampled$. On the event that there exists a region $U_1 \subseteq V \subseteq U$ as above, we can detect it with positive probability by sampling a pair of flow lines $\eta_u,\eta_{u,v}$ from random starting points where $\eta_{u,v}$ is reflected off $\eta_u$ in the opposite direction. The CLE configuration within $V$ is then determined by the values of the GFF in $V$. By the absolute continuity and conformal invariance of the GFF, the law of $\varphi(V)$ is absolutely continuous with respect to the law of a region detected by a pair of flow lines sampled in $\varphi(U)$. This shows that the law of the image of the CLE configuration within $V$ is absolutely continuous with respect to the law of the CLE configuration within $\varphi(U)$ in \emph{some} level of nesting. Applying again the resampling procedure from \cite[Section~5]{amy-cle-resampling} shows that it is absolutely continuous with respect to the law of the gasket $\varphi(V) \cap \Upsilon_\Gamma$.
\end{proof}

In the remainder of the section, we assume that $\varphi\colon U \to \wt{U}$ is fixed. Suppose $\mu\colon U \to \R_+$ is a continuous function (which we can think of as a weighted geometry). For an admissible path $\gamma$ within $U$, we write
\begin{align*}
\lmetconf{\mu}{\gamma} &\defeq \int \mu^\alpha d\lmet{\gamma} ,\\
\lmettconf{\mu}{\gamma} &\defeq \int \mu^\alpha d\lmett{\gamma} = \int (\mu\circ\varphi^{-1})^\alpha d\len{\Fd;\varphi(\Gamma)}{\varphi(\gamma)} ,
\end{align*}
where $\alpha$ is the exponent from Theorem~\ref{thm:exponent}.

We now state the main result of this subsection which is a more general version of Theorem~\ref{thm:conformal_covariance}.
\begin{theorem}\label{thm:conformal_covariance_gen}
Suppose that $U, \wt{U} \subseteq \D$ are simply connected domains, $\varphi \colon U \to \wt{U}$ is a conformal map, and $\mu\colon U \to \R_+$ is a continuous function. Then almost surely
\[ \lmettconf{\mu}{\gamma} = \lmetconf{\abs{\varphi'}\mu}{\gamma} \]
for each admissible path $\gamma\colon [0,1] \to U \cap \Upsilon_\Gamma$.
\end{theorem}

The proof of Theorem~\ref{thm:conformal_covariance_gen} follows the same strategy as Propositions~\ref{pr:bilipschitz} and~\ref{pr:uniqueness}. Recall the notation introduced in Section~\ref{se:setup_pf}. For each $z \in U$ and large $j \in \N$ we want to compare the metrics $\met{\cdot}{\cdot}{\varphi'(z)\Gamma}$ and $\mett{\cdot}{\cdot}{\varphi(\Gamma)}$ conditionally on $\CF_{z,j}$. We view them as metric spaces equipped with an embedding into the plane, and recall the GHf metric defined in~\eqref{eq:ghf}. More precisely, consider
\[
\Gamma_{z,j} = 2^j(\Gamma-z)
,\quad 
\Gamma^\varphi_{z,j} = 2^j \varphi'(z)^{-1}(\varphi(\Gamma)-\varphi(z))
\]
Note that
\begin{align*}
\met{\cdot}{\cdot}{\Gamma_{z,j}} &= 2^{j\alpha}\met{z+2^{-j}\cdot}{z+2^{-j}\cdot}{\Gamma} \\
\text{and}\quad \met{\cdot}{\cdot}{\Gamma^\varphi_{z,j}} &= (\abs{\varphi'(z)}^{-1}2^{j})^\alpha \met{\varphi(z)+\varphi'(z)2^{-j}\cdot}{\varphi(z)+\varphi'(z)2^{-j}\cdot}{\varphi(\Gamma)}
\end{align*}
by the scale covariance (Lemma~\ref{le:scaling_covariance}) and rotational invariance (Lemma~\ref{le:rotational_invariance}). 
The conformal transformation between $\Gamma_{z,j}$ and $\Gamma^\varphi_{z,j}$ is given by
\[
\varphi_{z,j} = 2^j \varphi'(z)^{-1} (\varphi(z+2^{-j}\cdot)-\varphi(z)) 
\]
which converges to the identity map as $j \to \infty$ with rate uniformly in $z \in U_1$.

The main step is to prove the following lemma.
\begin{lemma}\label{le:distorted_metric}
Let $\varphi \colon U \to \wt{U}$ be a conformal transformation, $U,\wt{U} \subseteq \D$, and $U_1 \Subset U$. For each $\epsilon > 0$, $\dsep > 0$ there is $j_0 \in \N$ such that
\[
\p\!\left[ \dGHf\left( \metres{B(0,1)}{\cdot}{\cdot}{\Gamma_{z,j}} , \metres{B(0,1)}{\cdot}{\cdot}{\Gamma^\varphi_{z,j}} \right) \ge \epsilon \mmiddle| \CF_{z,j} \right] \one_{\Esep_{z,j}} < \epsilon 
\]
for each $z \in U_1$, $j \ge j_0$.
\end{lemma}

The challenge in proving Lemma~\ref{le:distorted_metric} is that the function that maps $B(0,1) \cap \Upsilon_\Gamma$ to $\metres{B(0,1)}{\cdot}{\cdot}{\Gamma}$ is not necessarily continuous. To show Lemma~\ref{le:distorted_metric}, we apply Lusin's theorem to find a nearly continuous version of the function, and apply the continuity in total variation of the multichordal \clekp{} law (Proposition~\ref{pr:mccle_gasket_tv_convergence}).

\begin{proof}[Proof of Lemma~\ref{le:distorted_metric}]
The conditional law of the remainder of $\Gamma$ given $\CF_{z,j}$ is that of a multichordal \clekp{} in the marked domain $B(z,3\cdot 2^{-j})^{*,B(z,2\cdot 2^{-j})}$ conditionally on the link pattern $\alpha^*_{z,j}$, as recalled in Section~\ref{se:setup_pf}. By the conformal invariance of multichordal \clekp{}, the same is true for the conditional laws of $\Gamma_{z,j}$ and $\Gamma^\varphi_{z,j}$ which are multichordal \clekp{} in the transformed domains.

Let $F$ denote the measurable function that maps $\ol{B(0,1) \cap \Upsilon_\Gamma}$ to $\metres{B(0,1)}{\cdot}{\cdot}{\Gamma}$. Here we view $\ol{B(0,1) \cap \Upsilon_\Gamma}$ as metric spaces embedded in the plane, as described in \cite{amy2025tightness}. Let $\CS$ denote the space of compact metric spaces embedded in the plane, equipped with the GHf topology~\eqref{eq:ghf}.

Let $(D^*,\ul{x}^*)$ be a marked domain, $\alpha^*$ an internal link pattern, and consider the multichordal \clekp{} law in $(D^*,\ul{x}^*)$ conditioned on $\alpha^*$. By Lusin's theorem, for every $\epsilon > 0$ there is a compact subset $\CK \subseteq \CS$ with probability at least $1-\epsilon$ (with respect to this multichordal \clekp{} law) such that the restriction of $F$ to $\CK$ is uniformly continuous. We now explain that for fixed $\dsep > 0$, we can choose the set $\CK$ such that its probability is at least $1-\epsilon$ with respect to every multichordal \clekp{} law in an $\dsep$-separated marked domain with conformal radius between~$2$ and~$3$. Indeed, by Proposition~\ref{pr:mccle_gasket_tv_convergence}, the law of $B(0,1) \cap \Upsilon_\Gamma$ is continuous in total variation with respect to the marked domain $(D^*,\ul{x}^*)$. Due to Proposition~\ref{pr:link_probability}, the same holds for the conditional laws conditioned on the link pattern $\alpha^*$. Since the set of $\dsep$-separated marked domains with conformal radius between~$2$ and~$3$ is compact in the Carathéodory topology, the continuity is uniform on this set. Hence, we can choose the set $\CK$ above such that its probability is at least $1-\epsilon$ with respect to every such multichordal \clekp{} law.

Recall that $\varphi_{z,j}$ converges to the identity map with rate uniformly in $z \in U_1$. Therefore, if $j$ is large enough, on the event $\Esep_{z,j}$ the conditional probability given $\CF_{z,j}$ that both $\ol{B(0,1) \cap \Upsilon_{\Gamma_{z,j}}}$ and $\ol{B(0,1) \cap \Upsilon_{\Gamma^\varphi_{z,j}}}$ are in $\CK$ is at least $1-2\epsilon$. By the uniform continuity of the restriction of $F$ to $\CK$, we obtain the desired result.
\end{proof}

\begin{proof}[Proof of Theorem~\ref{thm:conformal_covariance_gen}]
Fix $U_1 \Subset U$. In the first step we show that there are deterministic constants $0 < c_1 \leq c_2 < \infty$ such that almost surely
\begin{equation}\label{eq:conf_cov_bilipschitz}
c_1 \lmettconf{\mu}{\gamma} \le \lmetconf{\abs{\varphi'}\mu}{\gamma} \le c_2 \lmettconf{\mu}{\gamma}
\end{equation}
for each admissible $\gamma \subseteq U_1$.

Let $b=8$ (say). The proofs of Proposition~\ref{pr:ball_crossing} and Lemma~\ref{le:len_lb_conditional} apply also to the metric $\mettres{U}{\cdot}{\cdot}{\Gamma}$, due to the conformal invariance of multichordal \clekp{}. Moreover, we have $\median[2^{-j}]{} = 2^{-j\alpha}\median[1]{}$ by Lemma~\ref{le:scaling_covariance}. Therefore, there is a constant $M$ such that for every sufficiently small $\delta > 0$ and $z \in U_1$ the probability is $1-O(\delta^b)$ that $9/10$ fraction of scales $j \in \{ \log_2(\delta^{-1}), \ldots, \log_2(\delta^{-2}) \}$ satisfy the following: For every admissible path $\gamma \subseteq B(z,2^{-j})$ that crosses $A_{z,j}$ we have
\begin{alignat}{2}
\metres{B(z,2^{-j+1})}{\gamma(0)}{\gamma(1)}{\Gamma} &\le M2^{-j\alpha} ,
&\qquad \lmet{\gamma} &\ge M^{-1}2^{-j\alpha} , \label{eq:conf_good_scale1}\\
\mettres{B(z,2^{-j+1})}{\gamma(0)}{\gamma(1)}{\Gamma} &\le M(\abs{\varphi'(z)}2^{-j})^\alpha ,
&\qquad \lmett{\gamma} &\ge M^{-1}(\abs{\varphi'(z)}2^{-j})^\alpha . \label{eq:conf_good_scale2}
\end{alignat}

Following the proof of Proposition~\ref{pr:bilipschitz}, we see that almost surely, for every admissible path $\gamma \subseteq U_1$ we can find partitions $(t_i)$ with arbitrarily fine mesh size such that
\[ \sum_i \abs{\varphi'(\gamma(t_i))}^\alpha \mu(\gamma(t_i))^\alpha \, \metres{U}{\gamma(t_i)}{\gamma(t_{i+1})}{\Gamma} \le M^2 \lmettconf{\mu}{\gamma} , \]
implying that
\[ \lmetconf{\abs{\varphi'}\mu}{\gamma} \le M^2 \lmettconf{\mu}{\gamma} . \]
Similarly, we find that
\[ \lmettconf{\mu}{\gamma} \le M^2 \lmetconf{\abs{\varphi'}\mu}{\gamma} . \]
This shows~\eqref{eq:conf_cov_bilipschitz}.

In the second step we show that~\eqref{eq:conf_cov_bilipschitz} holds with $c_1 = c_2 = 1$. Here we additionally use Lemma~\ref{le:distorted_metric}. 
Let $\epsilon > 0$. By the same independence across scales argument, we see that for every sufficiently small $\delta > 0$ and $z \in U_1$ the probability is $1-O(\delta^b)$ that $9/10$ fraction of scales $j \in \{ \log_2(\delta^{-1}), \ldots, \log_2(\delta^{-2}) \}$ satisfy
\begin{equation}\label{eq:conf_good_scale3}
\sup_{x \in \bIn A_{z,j},\ y \in \bOut A_{z,j}} \abs{\met{x}{y}{\Gamma_{z,j}} - \met{x}{y}{\Gamma^\varphi_{z,j}}} < \epsilon . 
\end{equation}
On a scale $j$ where both~\eqref{eq:conf_good_scale1},~\eqref{eq:conf_good_scale2}, and~\eqref{eq:conf_good_scale3} hold, we have
\begin{alignat*}{2}
&&\abs{\varphi'(z)}^\alpha \mu(z)^\alpha \metres{B(z,2^{-j+1})}{\gamma(0)}{\gamma(1)}{\Gamma} &\le (1+\epsilon) \mu(z)^\alpha \lmett{\gamma} \\
\text{and}\quad&& \mu(z)^\alpha \mettres{B(z,2^{-j+1})}{\gamma(0)}{\gamma(1)}{\Gamma} &\le (1+\epsilon)\abs{\varphi'(z)}^\alpha \mu(z)^\alpha \lmet{\gamma}
\end{alignat*}
for every admissible path $\gamma \subseteq \ol{A}_{z,j}$ that crosses $A_{z,j}$.

Following the proof of Proposition~\ref{pr:uniqueness}, we conclude that assuming~\eqref{eq:conf_cov_bilipschitz} with $c_1 < 1 < c_2$, we can find $\epsilon' > 0$ such that~\eqref{eq:conf_cov_bilipschitz} holds with $c_1+\epsilon'$ and $c_2-\epsilon'$. This show that~\eqref{eq:conf_cov_bilipschitz} in fact holds with $c_1 = c_2 = 1$.
\end{proof}

\subsection{General domains and the whole-plane \clekp{} metric}
\label{se:whole_plane_metric}

In this subsection we prove the Theorems~\ref{th:whole_plane_metric} and~\ref{th:metric_general_domain}.

\begin{proof}[Proof of Theorem~\ref{th:whole_plane_metric}]
\emph{Step 1.  Construction.} Let $\CQ$ be the collection of Jordan domains $U$ with $\diam(U) < 1/2$ and whose closure is a union of closed squares with corners in $2^{-n}\Z^2$ for some $n \in \N$. For each $U \in \CQ$, we let $\CK_U$ be the collection of connected components of $\Upsilon_i \cap U$ for each cluster $\Upsilon_i$. For $U \in \CQ$, $z \in \Q^2$, let $K_{U,z} \in \CK_U$ be the outermost component in $\CK_U$ that disconnects~$z$ from $\partial U$. The internal metric on $K_{U,z}$ is defined by translation invariance and absolute continuity (see Section~\ref{sec:conformally_covariant} for details), and by Theorem~\ref{thm:uniqueness} is a measurable function of $\Gamma^\C\big|_U$. Moreover, for each $(U,z)$ and $(U',z')$, it holds a.s.\ that if $\gamma$ is a path contained in $K_{U,z}$ and $K_{U',z'}$, then its length with respect to the two internal metrics agree. This is a consequence of the locality applied to (each component of) $U \cap U'$ and Lemma~\ref{le:internal_metrics_compatibility}. This uniquely defines the length of each path, and the associated length metric $\met{\cdot}{\cdot}{\Gamma^\C}$. Finally, the translation invariance $\met{\cdot}{\cdot}{\Gamma^\C}$ just follows by the translation invariance of the metric in the setup of Theorem~\ref{thm:uniqueness}. The scale covariance and conformal covariance follow from Theorems~\ref{thm:exponent} and~\ref{thm:conformal_covariance}.

\emph{Step 2.  Uniqueness.} Given any such metric $\met{\cdot}{\cdot}{\Gamma^\C}$, we obtain a geodesic \clekp{} metric in the sense considered in Theorem~\ref{thm:uniqueness}. Let $\CL \in \Gamma^\C$ be the outermost loop contained in $\D$, and let $\Upsilon_{i_\CL}$ be the cluster bounded by $\CL$. By the same argument as above, the law of $\CL$ is mutually absolutely continuous with respect to the loop considered in the setup of Theorem~\ref{thm:uniqueness}, and hence the metric on $\Upsilon_{i_\CL}$ induces a metric that satisfies the assumptions in Theorem~\ref{thm:uniqueness}. This shows that at least the metric on $\Upsilon_{i_\CL}$ is uniquely determined up to a deterministic factor. But by Part~1 of the proof, we uniquely recover the full metric $\met{\cdot}{\cdot}{\Gamma^\C}$ from its restriction to $\Upsilon_{i_\CL}$.
\end{proof}

\begin{proof}[Proof of Theorem~\ref{th:metric_general_domain}]
 For each $U \Subset D \subseteq \C$, the laws of $\Gamma^D\big|_U$ and of $\Gamma^\C\big|_U$ are mutually absolutely continuous (by the same argument as given in the proof of Lemma~\ref{lem:conf_map_abs_cont}). Therefore, in view of Lemma~\ref{le:internal_metrics_compatibility}, we can define the metric $\met{\cdot}{\cdot}{\Gamma^D}$ from the whole-plane \clekp{} metric $\met{\cdot}{\cdot}{\Gamma^\C}$ in Theorem~\ref{th:whole_plane_metric}, and the desired properties follow from Theorem~\ref{th:whole_plane_metric}. Conversely, if we are given a collection of metrics as in Theorem~\ref{th:metric_general_domain}, this determines a whole-plane \clekp{} metric as in Theorem~\ref{th:whole_plane_metric}. The uniqueness thus follows from the uniqueness statement in Theorem~\ref{th:whole_plane_metric}.
 \end{proof}

\bibliographystyle{alpha}
\input{cle_geodesic_uniqueness.bbl}

\end{document}

%% file: cle_geodesic_uniqueness.bbl
\newcommand{\etalchar}[1]{$^{#1}$}
\providecommand{\noopsort}[1]{}